\documentclass{amsart}

\usepackage[utf8]{inputenc}
\usepackage[english]{babel}

\usepackage{fullpage}

\usepackage{amsmath}
\usepackage{amsfonts}
\usepackage{amsthm}
\usepackage{amssymb}
\usepackage{theoremref}
\usepackage{colonequals}
\usepackage{url}
\usepackage{csquotes}
\usepackage{mathrsfs}
\usepackage{graphicx}
\usepackage{tabularray}
\usepackage{accents}

\newcounter{master}
\numberwithin{master}{section}

\theoremstyle{plain}

\newtheorem{theorem}[master]{Theorem}
\newtheorem{proposition}[master]{Proposition}
\newtheorem{corollary}[master]{Corollary}
\newtheorem{question}{Question}
\newtheorem{intro}{Theorem}

\theoremstyle{definition}

\newtheorem{definition}[master]{Definition}
\newtheorem*{notation}{Notation}

\theoremstyle{remark}

\newtheorem*{remark}{Remark}

\makeatletter
\let\c@equation\c@master

\let\c@figure\c@master

\let\c@table\c@master

\makeatother

\makeatletter
\def\@cite#1#2{\textup{[\textbf{#1}\if@tempswa , #2\fi]}}
\def\@biblabel#1{[\textbf{#1}]}
\makeatother

\DeclareMathOperator{\Const}{const}
\DeclareMathOperator{\Conv}{conv}
\DeclareMathOperator{\Vol}{vol}

\makeatletter\def\tr#1{{\reset@font[}#1{\reset@font]}}\makeatother

\def\sref#1{Section~\S$\ref{#1}$}
\def\fref#1{fig.~$\ref{#1}$}

\UseTblrLibrary{booktabs}

\def\dbar{\mathchar'26\mkern-9mu\delta}

\begin{document}

\title{On the homothety conjecture for the body of flotation and the body of buoyancy on a plane}

\author[B. Zawalski]{Bartłomiej Zawalski}
\address[B. Zawalski]{Kent State University, Kent, OH, USA}
\email{bzawalsk@kent.edu}
\thanks{The author is supported in part by U.S. National Science Foundation Grants DMS-1900008 and DMS-2247771.}
\subjclass[2010]{Primary 52A10; Secondary 51N10, 53A04, 53A15, 70C20}
\keywords{homothety conjecture, floating body problem, body of flotation, body of buoyancy, body of illumination}

\begin{abstract}
We investigate several closely related \enquote{homothety conjectures} for convex bodies on a plane. Using the modern language of differential geometry, we systematically derive the fundamental properties of bodies of flotation, bodies of buoyancy, and bodies of illumination. As a direct consequence of our results, we show that if the body of flotation is homothetic to the body of buoyancy, and if every chord of flotation cuts off from the boundary exactly $\frac{1}{3}$ of its total affine arc length, then $K$ is an ellipse. We also provide natural affine counterparts of the classical theorems on the floating body problem from the Scottish Book due to H. Auerbach. In particular, we obtain an affine counterpart of Zindler carousels introduced by J. Bracho, L. Montejano, and D. Oliveros.
\end{abstract}

\maketitle

\tableofcontents

\section{Introduction}

Let $K\subset\mathbb R^n$ be a convex body, i.e., a compact convex set with non-empty interior. For any $\delta\in(0,\Vol_n(K))$, the convex body of flotation of $K$ is the set of points that remain above the water level when a solid body of shape $K$ and uniform density $\delta/\Vol_n(K)$ floats in any orientation.\\

Floating bodies naturally emerge in such branches of mathematics as affine-invariant geometry \cite{SchuttWerner,Werner_1994}, approximation by polytopes \cite{Barany_Larman_1988,MR4420518}, and statistics \cite{NagySW}, to mention just a few. Even though they have been known since the ancient works of Archimedes, they are still much in demand, if only because of their connection with the symmetry of hyperplanar sections \cite{10.1093/imrn/rnx211,Olov}.\\

The following open problem concerning floating bodies has been extensively studied in the literature:

\begin{question}[Sch\"utt-Werner \cite{Schutt1994}]\thlabel{que:01}
Let $K\subset\mathbb R^n$ be a convex body and let $\delta\in(0,\Vol_n(K))$. Suppose that the body of flotation $F_\delta(K)$ and $K$ are homothetic. Is $K$ an ellipsoid?
\end{question}

\noindent Here, we say that two sets are homothetic if there exists a point and a dilation with center at this point that maps one set to the other. From now on, we will refer to \thref{que:01} as the \emph{homothety conjecture}. To the author's best knowledge, the homothety conjecture was first formulated by C. Sch\"utt and E.M. Werner \cite{Schutt1994}, and was motivated by the work of M. Meyer and S. Reisner \cite{Meyer1989}, generalizing Blaschke's characterization of ellipsoids \cite[Anhang.VII]{BlaschkeKK}. In their proof, the authors used the fact that $K$ is an ellipsoid provided that for every $\delta\in(0,\Vol_n(K))$ the body of flotation $F_\delta(K)$ is homothetic to $K$ with respect to the same center of homothety (cf. \cite[Lemma~3]{Meyer1989}). This seemingly undue assumption was subsequently relaxed. Firstly, C. Sch\"utt and E.M. Werner \cite{Schutt1994} showed that $K$ is an ellipsoid provided that there exists a sequence $\delta_k\to 0^+$ such that for every $k\in\mathbb N$ the body of flotation $F_{\delta_k}(K)$ is homothetic to $K$ with respect to the same center of homothety. Secondly, it follows from the breakthrough work of A. Stancu \cite{Stancu,Stancu_2009} that the assumption that the center of homothety is the same for every $k\in\mathbb N$ can be dropped if the boundary of $K$ is of class $C^2_+$. Finally, E.M. Werner and D. Ye \cite{WernerYe} removed the smoothness assumption, and showed furthermore that $K$ with boundary of class $C^3$ is an ellipsoid provided that there exists a single sufficiently small $\delta<\delta_*(K)$ such that the body of flotation $F_\delta(K)$ is homothetic to $K$. Recently, M. Angeles Alfonseca, F. Nazarov, D. Ryabogin, A. Stancu, and V. Yaskin \cite{ANRSY} showed that on a plane, the homothety conjecture is false and constructed counterexamples for densities $\delta/\Vol_2(K)$ in some countable set with accumulation point at $\frac{1}{2}$. On the other hand, they proved that the conjecture holds for any density $\delta/\Vol_n(K)\in(0,1)$, provided that $K$ is symmetric and sufficiently close to the Euclidean ball.\\


It was a breakthrough idea of C. Dupin reflected in his memoir \cite{dupin1822applications} that any study of flotation and stability must consider three inextricably related bodies: the body of flotation, the body of buoyancy, and the body $K$ itself. Therefore, the following questions in the spirit of \thref{que:01} are of independent interest:

\begin{question}\thlabel{que:02}
Let $K\subset\mathbb R^n$ be a convex body, and let $\delta\in(0,\Vol_n(K))$ be such that the body of buoyancy $B_\delta(K)$ and $K$ are homothetic. Is $K$ an ellipsoid?
\end{question}

\begin{question}\thlabel{que:03}
Let $K\subset\mathbb R^n$ be a convex body, and let $\delta\in(0,\Vol_n(K))$ be such that the body of flotation $F_\delta(K)$ and the body of buoyancy $B_\delta(K)$ are homothetic. Is $K$ an ellipsoid?
\end{question}

It is worth noting that all three \thref{que:01,que:02,que:03} are closely related to the classical open problems formulated by H. Busemann and C.M. Petty \cite{BusemannPetty}. Moreover, if $K\subset\mathbb R^2$ is a symmetric convex body with density $\delta/\Vol_2(K)=\frac{1}{2}$, then \thref{que:01} becomes simply equivalent to \cite[Problem~5]{BusemannPetty}, and \thref{que:02} becomes simply equivalent to \cite[Problem~8]{BusemannPetty}, although in both cases this observation is by no means trivial. In comments that followed, the authors of \cite{BusemannPetty} claimed that the answer to Problem 5 on a plane is already known to be negative, while the answer to Problem 8 on a plane is affirmative. However, to the author's best knowledge, the latter proof has never been published. An independent proof may be found, e.g., in the paper of V. Balestro, H. Martini, and E. Shonoda \cite[Theorem~6.2]{BALESTRO2019347}.\\

From the point of view of the idea first brought up by M. Angeles Alfonseca, F. Nazarov, D. Ryabogin, and V. Yaskin in an expository paper \cite{ANRY} and then applied to the homothety conjecture by the same authors along with A. Stancu in \cite{ANRSY}, both \thref{que:01,que:02} may be viewed as fixed-point problems for certain contraction operators. Hence, we strongly believe that an approach similar to the one that has already proven effective for \thref{que:01} may also be applied to solve \thref{que:02}, provided that $K$ is sufficiently close to the Euclidean ball. Interestingly, we learned from personal communication with V. Shaw that, unlike \thref{que:01}, the answer to \thref{que:02} is positive for any $K$ in this regime. But \thref{que:03} already eludes the method of \cite{ANRSY}, because in that case contraction operators appear on both sides of equality. In this paper, we fill this very gap. Our main results are as follows:

\begin{intro}
Let $K\subset\mathbb R^2$ be a convex body with boundary of class $C^2$, and let $\delta\in(0,\Vol_n(K))$. If $F_\delta(K)$ is homothetic to $B_\delta(K)$, then the affine distance between the endpoints of every chord of flotation is constant.
\end{intro}

\begin{intro}
Let $K\subset\mathbb R^2$ be a convex body with boundary of class $C^2$, and let $\delta\in(0,\Vol_n(K))$. If $F_\delta(K)$ is homothetic to $B_\delta(K)$, then the affine arc length of the boundary cut off from $K$ by every chord of flotation is constant.
\end{intro}

\begin{intro}
Let $K\subset\mathbb R^2$ be a convex body with boundary of class $C^2$, and let $\delta\in(0,\Vol_n(K))$. If $F_\delta(K)$ is homothetic to $B_\delta(K)$, and if $\delta$ is such that every chord of flotation cuts off from the boundary exactly $\frac{1}{3}$ of its total affine arc length, then $K$ is an ellipse.
\end{intro}

\noindent The smoothness assumption seems to be superfluous, and we believe it can be relaxed, especially since the problem has an inherent smoothing property. However, it would most probably require a long-winded, technical argument, which would overshadow the main idea of the paper.\\

\begin{figure}
\includegraphics{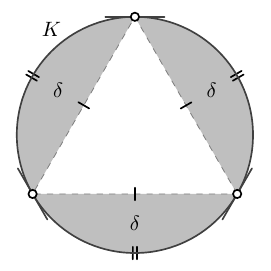}
\caption{A drawing illustrating the main theorems}
\label{fig:08}
\end{figure}

The above theorems are similar in spirit to the results of H. Auerbach \cite{Auerbach1938} and J. Bracho, L. Montejano, and D. Oliveros \cite{Bracho2004}, closely related to the floating body problem:

\begin{question}[{\cite[19. Problem: Ulam]{scottish}}]
Is a solid of uniform density which will float in water in every position a sphere?
\end{question}

\noindent H. Auerbach showed that on a plane, the length of every chord of flotation is constant, and the arc length of the boundary cut off from $K$ by every chord of flotation is likewise constant. Using his work, J. Bracho, L. Montejano, and D. Oliveros showed that except for the circle, there are no figures that float in equilibrium in every orientation with perimetral densities $\frac{1}{3}$ and $\frac{1}{4}$, and sketched the argument for perimetral densities $\frac{1}{5}$ and $\frac{2}{5}$. Our conclusions, however, follow from a completely different set of assumptions. Importantly, while the floating body problem is covariant under the action of the Euclidean group $\mathbb R^2\rtimes\mathrm{O}(\mathbb R^2)$, the homothety conjecture is covariant under the action of the general affine group $\mathbb R^2\rtimes\mathrm{GL}(\mathbb R^2)$. Thus, our theorems provide natural counterparts of results from \cite{Auerbach1938,Bracho2004} in the affine setting.\\

The purpose of this paper is twofold. Firstly, we systematically derive all the fundamental results for bodies of flotation and bodies of illumination, and present them all together in a coherent way. Secondly, we prove several new theorems related to the homothety conjecture as formulated in \thref{que:03}. Although the formulas for bodies of flotation are mostly classical and can be found in many publications on the subject, the analogous formulas for bodies of illumination are hard to find anywhere, and some of them we could not find at all. Our purely algebraic approach allows us to easily obtain formulas constraining higher-order derivatives, which is usually not the case with the commonly used geometric approach. \thref{que:03} emerges naturally in studies concerning the symmetry of hyperplanar sections, but seems not to have been considered so far.\\

The structure of the paper is the following: In \sref{sec:01}, for the reader's convenience, we briefly introduce some basic concepts from differential geometry and fluid statics, playing an important role in our argument. In \sref{sec:05}, we construct a parametrization of the body of flotation, the body of illumination, and the corresponding centroid curves, and we derive all necessary formulas in a uniform, systematic way. Finally, in \sref{sec:06}, we give the complete proofs of the new results.

\section{Definitions and basic concepts}\label{sec:01}

In this section, we recall several classical results from differential geometry and fluid statics that we will use in our proofs. We also refer an interested reader to \cite{do2016differential} for \sref{sec:02}, to \cite[\S I.1]{nomizu1994affine} and \cite[\S I.3]{Buchin} for \sref{sec:03}, and to \cite[Ch.~XXIV]{de1924lecons} and \cite{ARSY} for \sref{sec:04}. Adopting the standard convention in differential geometry textbooks, for notational simplicity, we will omit the function parameters when they are clear from the context.

\subsection{Riemannian differential geometry}\label{sec:02}

Denote by $\mathbb R^2$ the set of pairs $(x,y)$ of real numbers. We say that a real function is \emph{of class $C^k$} if it has, at all points, a continuous derivative of order $k$. A \emph{parametrized differentiable curve} is a differentiable map $\gamma:I\to\mathbb R^2$ of an open interval $I=(a,b)$ of the real line $\mathbb R$ into $\mathbb R^2$. The vector $\gamma'(s)=\mathrm d\gamma/\mathrm ds$ is called the \emph{tangent vector} of the curve $\gamma$ at $s$. For more details, see \cite[\S 1-2]{do2016differential}.\\

Let $\boldsymbol u=(u_1,u_2)$ and $\boldsymbol v=(v_1,v_2)$ belong to $\mathbb R^2$, and let $\vartheta$, $0\leq\vartheta\leq\pi$ be the oriented angle between $\boldsymbol u$ and $\boldsymbol v$. Denote the \emph{norm} (or \emph{length}) of $\boldsymbol u$ by $|\boldsymbol u|$. The \emph{inner product} $\boldsymbol u\cdot\boldsymbol v$ is defined by $\boldsymbol u\cdot\boldsymbol v=u_1v_1+u_2v_2$ and is equal to $|\boldsymbol u||\boldsymbol v|\cos\vartheta$. The \emph{determinant} $\det(\boldsymbol u,\boldsymbol v)$ is defined by $\det(\boldsymbol u,\boldsymbol v)=u_1v_2-u_2v_1$ and is equal to $|\boldsymbol u||\boldsymbol v|\sin\vartheta$. For more details, see \cite[\S 1-2]{do2016differential}. Vector symbols will be typeset in bold face so that they can be clearly distinguished.\\

A parametrized differentiable curve $\gamma:I\to\mathbb R^2$ is said to be \emph{regular} if $\gamma'(s)\neq\boldsymbol 0$ for all $s\in I$ \cite[\S 1-3]{do2016differential}. Let $\gamma:I\to\mathbb R^2$ be a curve parametrized by arc length $s$. Denote by $\boldsymbol t(s)\colonequals\gamma'(s)$ the \emph{unit tangent vector} of $\gamma$ at $s$. The number $\kappa(s)\colonequals|\gamma''(s)|$ is called the \emph{curvature} of $\gamma$ at $s$. At points where $\kappa(s)\neq 0$, the \emph{unit normal vector} is well-defined by the formula
\begin{equation}\label{eq:28}\dot{\boldsymbol t}(s)=\kappa(s)\boldsymbol n(s).\end{equation}
For more details, see \cite[\S 1-5]{do2016differential}.\\

In the case of a plane curve $\gamma:I\to\mathbb R^2$, it is possible to give the curvature $\kappa$ a sign. If $\gamma:I\to\mathbb R^2$ is a regular parametrized curve of class $C^2$, the \emph{oriented curvature} of $\gamma$ at $s$ is defined by
\begin{equation}\label{eq:19}\kappa=\frac{\det(\gamma',\gamma'')}{|\gamma'|^3}\end{equation}
(see \cite[\S 1-5, Remark~1 and Exercise~12]{do2016differential}).\\

A \emph{closed plane curve} of class $C^k$ is a regular parametrized curve $\gamma:[a,b]\to\mathbb R^2$ such that $\gamma$ and all its derivatives agree at $a$ and $b$. The curve $\gamma$ is \emph{simple} if it has no further self-intersections. A simple closed curve $\gamma$ on a plane bounds a region of this plane that is called the \emph{interior} of $\gamma$. It follows from Jordan curve theorem \cite[\S 5-6, Theorem~1]{do2016differential}.\\

The negative derivative $\boldsymbol S(\boldsymbol v)=-D_{\boldsymbol v}\boldsymbol n$ of the unit normal vector field of a surface is an extrinsic curvature, known as the \emph{shape operator} or \emph{Weingarten map}. The eigenvalues of $\boldsymbol S$ correspond to the principal curvatures of the surface, and the eigenvectors are the corresponding principal directions. In particular, the Gaussian curvature is given by the determinant of $\boldsymbol S$ (see \cite[\S 3-3]{do2016differential}).

\subsection{Affine differential geometry}\label{sec:03}

In this section, we will outline the geometric properties of plane curves expressed by conditions invariant under the group of unimodular affine transformations $\mathbb R^2\rtimes\mathrm{SL}(\mathbb R^2)$. Let's consider a smooth curve $\boldsymbol x=\boldsymbol x(s)$. We denote by $\dot{\boldsymbol x}$ the tangent vector $\mathrm d\boldsymbol x/\mathrm ds$. A given point and a given direction herein form a \emph{linear element}. Two linear elements $(\boldsymbol x,\dot{\boldsymbol x})$ and $(\boldsymbol y,\dot{\boldsymbol y})$ with non-parallel directions determine a triangle with signed area
$$T\colonequals\frac{1}{2}\frac{\det(\dot{\boldsymbol x},\boldsymbol y-\boldsymbol x)\det(\boldsymbol y-\boldsymbol x,\dot{\boldsymbol y})}{\det(\dot{\boldsymbol x},\dot{\boldsymbol y})}.$$
The signed \emph{affine distance} $\|\boldsymbol y-\boldsymbol x\|$ of two linear elements $(\boldsymbol x,\dot{\boldsymbol x})$ and $(\boldsymbol y,\dot{\boldsymbol y})$ is defined as
$$\|\boldsymbol y-\boldsymbol x\|\colonequals 2T^{\frac{1}{3}}$$
(cf. \cite[\S I.3]{Buchin}).\\

We say that the curve is \emph{non-degenerate} if it satisfies the condition $\det(\dot{\boldsymbol x},\ddot{\boldsymbol x})\neq 0$. A non-degenerate curve $\boldsymbol x$ admits a parameter $\sigma$ such that
\begin{equation}\label{eq:09}\det(\boldsymbol x',\boldsymbol x'')=1,\end{equation}
where $'$ denotes the derivative with respect to $\sigma$. Such a parameter, called the \emph{affine arc length parameter}, is unique up to a constant. Namely,
\begin{equation}\label{eq:07}\sigma(s)=\int_{s_0}^s\det(\dot{\boldsymbol x},\ddot{\boldsymbol x})^{\frac{1}{3}}\;\mathrm ds.\end{equation}
By differentiating \eqref{eq:09} we have $\det(\boldsymbol x',\boldsymbol x''')=0$. Hence, the two vectors are linearly dependent, i.e., as $\boldsymbol x'\neq 0$, $\boldsymbol x'''+k\boldsymbol x'=0$, where $k(\sigma)$ denotes a scalar quantity. The last equation yields
$$k(\sigma)=\det(\boldsymbol x'',\boldsymbol x''').$$
We call $k$ the \emph{affine curvature}, and $\boldsymbol x''$ the \emph{affine normal vector}. Note that the affine arc length parameter and the affine curvature are invariant, and the affine normal vector is covariant under an equi-affine transformation of a plane (cf. \cite[\S I.1]{nomizu1994affine} and \cite[\S I.3]{Buchin}).\\

A Blaschke hypersurface is called an \emph{improper affine hypersphere} if all the affine normals are parallel. If all the affine normals meet at one point, then the hypersurface is called a \emph{proper affine hypersphere} (cf. \cite[Definition~II.3.3 and Proposition~II.3.5]{nomizu1994affine}). The following classical result is due to W. Blaschke and A. Deicke:
\begin{theorem}[{\cite[Theorem~III.7.5]{nomizu1994affine}}]\thlabel{thm:04}
If a connected, compact, non-degenerate hypersurface is either a proper or improper affine hypersphere, then it must be an ellipsoid.
\end{theorem}

\subsection{Bodies of flotation and illumination}\label{sec:04}

The body of flotation was originally introduced by C. Dupin in \cite{dupin1822applications} as part of his study on the theory of flotation of bodies by means of differential geometry. A new definition was given by I. B\'ar\'any and D.G. Larman \cite{Barany_Larman_1988} and independently by C. Sch\"utt and E.M. Werner \cite{SchuttWerner}, who introduced a \emph{convex body of flotation}. The body of illumination was introduced by E.M. Werner in \cite{Werner_1994}. Although the definitions of the body of flotation and the body of illumination seem to suggest some kind of duality relation, no equality between bodies of flotation of $K,K^*$ and bodies of illumination of $K,K^*$ can be achieved in general (cf. \cite[\S 1.2]{Mordhorst2020}). Nevertheless, in all our theorems, the body of flotation and the body of illumination play symmetrical roles and can be interchanged, still leaving the theorems true.

\begin{figure}
\includegraphics{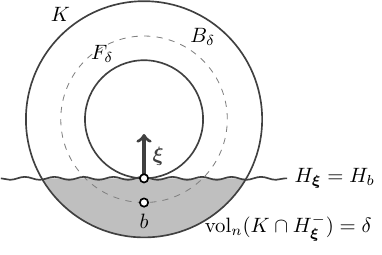}
\caption{Illustrative drawing of a floating body}
\label{fig:04}
\end{figure}

\subsubsection{Bodies of flotation}

Given a convex body $K\subset\mathbb R^n$ and $\delta\in(0,\Vol_n(K))$, we consider the family of hyperplanes $\{H_{\boldsymbol\xi}\}_{\boldsymbol\xi\in\mathbb S^{n-1}}$ such that the half-spaces $H_{\boldsymbol\xi}^-$ determined by $H_{\boldsymbol\xi}$ cut off from $K$ a volume $\delta$. From the point of view of fluid statics, each hyperplane in such a family corresponds to the water surface for a specific orientation in which the body $K$ is floating, with $\delta$ being the volume of the underwater part. The \emph{hypersurface of flotation} of $K$ is defined as the envelope (see \cite[\S 21]{weatherburn1955differential}) of the hyperplanes $H_{\boldsymbol\xi}$, and the \emph{body of flotation} of $K$ is the body whose boundary is the hypersurface of flotation. We will denote the hypersurface of flotation by $\Pi_\delta(K)$, and the body of flotation by $F_\delta(K)$.

\begin{remark}
There are instances in which the hypersurface of flotation does not enclose a convex body. For this reason, an alternative definition was introduced independently in \cite{Barany_Larman_1988} and \cite {SchuttWerner}. The \emph{convex body of flotation} is defined as the intersection
$$K_\delta\colonequals\bigcap_{\boldsymbol\xi\in\mathbb S^{n-1}}H_{\boldsymbol\xi}^+,$$
where $H_{\boldsymbol\xi}^+$ is the half-space complementary to $H_{\boldsymbol\xi}^-$. As an intersection of half-spaces, $K_\delta$ is always a convex body, and it coincides with $F_\delta(K)$ whenever the envelope of the water surfaces encloses a convex body.
\end{remark}

For each water surface $H_{\boldsymbol\xi}$, the centroid of the underwater part $K\cap H_{\boldsymbol\xi}^-$ is known in fluid statics as the \emph{center of buoyancy}. The \emph{hypersurface of buoyancy} is defined as the geometric locus of all the centers of buoyancy and denoted by $\Gamma_\delta(K)$. The hypersurface of buoyancy encloses a strongly convex body $B_\delta(K)$, called the \emph{body of buoyancy}. For convenience, given a center of buoyancy $b\in\Gamma_\delta(K)$, we will refer to the corresponding water surface as $H_b$ instead of $H_{\boldsymbol\xi}$.

\begin{remark}
The function
$$h(\boldsymbol\xi)=\frac{1}{\Vol_n(K)}\int_K|x\cdot\boldsymbol\xi|\;\mathrm dx$$
happens to be a supporting function of a convex body. This body is called the \emph{centroid body} of $K$ and its boundary is called the \emph{centroid hypersurface} of $K$ (cf. \cite[\S 3]{pjm/1103036936}). If $K$ is symmetric, then the centroid body of $K$ coincides with the body of buoyancy of $K$ for density $\delta/\Vol_n(K)=\frac{1}{2}$. Slightly abusing the language, we will use the name \enquote{centroid hypersurface} in the broader context of a hypersurface of buoyancy.
\end{remark}

Dupin was the first to use methods of differential geometry to study stability. His memoir \cite{dupin1822applications} contains three fundamental theorems that govern the flotation of bodies.

\begin{theorem}[First Theorem of Dupin {\cite[XXIV.2.476]{de1924lecons}}]\thlabel{Dup:01}
The body of buoyancy $B_\delta(K)$ is strongly convex, and the plane tangent to $B_\delta(K)$ at a point $b\in\Gamma_\delta(K)$ is parallel to the corresponding water surface $H_b$.
\end{theorem}

\noindent In fact, the boundary of $B_\delta(K)$ is always smoother than the boundary of $K$, which was proved much later by H. Huang, B.A. Slomka, and E.M. Werner:

\begin{theorem}[{\cite[Theorem~1.2]{HSW}}]\thlabel{thm:09}
Let $K\subset\mathbb R^n$ be a convex body. Suppose that its boundary $\partial K$ is of class $C^k$ for some $k\geq 0$. Then, for any $\delta\in(0,\Vol_n(K))$, the boundary $\Gamma_\delta(K)$ of its body of buoyancy is of class $C^{k+1}$.
\end{theorem}

\begin{theorem}[Second Theorem of Dupin {\cite[XXIV.2.477]{de1924lecons}}]\thlabel{Dup:02}
The body of flotation $F_\delta(K)$ touches each water surface $H_{\boldsymbol\xi}$ at the centroid of section $K\cap H_{\boldsymbol\xi}$.
\end{theorem}

\noindent A more general two-way variant of the Second Theorem of Dupin was proved later by S.P. Olovyanishnikov (cf. \cite[Lemma and subsequent Zamechaniye~1. and Zamechaniye~2.]{Olov}):

\begin{theorem}[{\cite[Theorem~3]{ryab}}]\thlabel{thm:19}
Let $K\subset\mathbb R^n$ be a convex body, and let $\{H_{\boldsymbol\xi}\}_{\boldsymbol\xi\in\mathbb S^{n-1}}$ be a family of hyperplanes. The half-spaces $H_{\boldsymbol\xi}^-$ determined by $H_{\boldsymbol\xi}$ cut off from $K$ the same volume if and only if the envelope touches each hyperplane $H_{\boldsymbol\xi}$ at the centroid of section $K\cap H_{\boldsymbol\xi}$.
\end{theorem}

\begin{theorem}[Third Theorem of Dupin {\cite[XXIV.2.478]{de1924lecons}}]\thlabel{Dup:03}
The shape operator $\boldsymbol S$ corresponding to a point $b$ of the hypersurface of buoyancy is equal to $\delta\boldsymbol I^{-1}$, where $\delta$ is the volume of the underwater part of $K$ and $\boldsymbol I$ is the inertia tensor of $K\cap H_b$, whose entries are $\boldsymbol I_{ij}\colonequals\int_{K\cap H_b}x_ix_j\;\mathrm dx$.
\end{theorem}

\noindent Recall that the Gaussian curvature is given by the determinant of $\boldsymbol S$ and thus is equal to
$$\frac{\delta^{n-1}}{L_{K\cap H_b}^{2n-2}\Vol_{n-1}(K\cap H_b)^{n+1}},$$
where $L_{K\cap H_b}$ is the isotropic constant of $K\cap H_b$. On a plane, the Third Theorem of Dupin can be reformulated in a much simpler form (cf. \cite[p.~23]{Zhuk} and \cite[\S 2]{Auerbach1938}), which will be enough for our purposes:

\begin{corollary}[Third Theorem of Dupin on a plane]
The curvature of $\Gamma_\delta(K)$ at almost every point $b\in\Gamma_\delta(K)$ is equal to
\begin{equation}\label{eq:29}\kappa(b)=\frac{12\delta}{\Vol_1(K\cap H_b)^3}.\end{equation}
\end{corollary}

\subsubsection{Bodies of illumination}

Given a convex body $K\subset\mathbb R^n$ and $\delta\in(0,\Vol_n(K))$, the \emph{hypersurface of illumination} of $K$ is the set of points $x\in\mathbb R^n$ such that the silhouette cone $\Conv(K\cup\{x\})\setminus K$ has volume $\delta$, and the \emph{body of illumination} of $K$ is the body whose boundary is the hypersurface of illumination. We will denote the hypersurface of illumination by $\Pi^\delta(K)$, and the body of illumination by $I^\delta(K)$. The \emph{centroid hypersurface} is defined as the geometric locus of all the centroids of silhouette cones with apexes on the hypersurface of illumination. We will denote the centroid hypersurface by $\Gamma^\delta(K)$.\\

The hypersurface of illumination always encloses a convex body \cite[Proposition~2(i)]{Werner_2006}. The same argument shows that the centroid hypersurface also encloses a convex body. However, unlike in the case of bodies of flotation, it does not need to be strictly convex.

\section{Auxiliary results}\label{sec:05}

Along the lines of A.Yu. Davidov \cite[p.~21]{Zhuk}, we begin by introducing the tools from differential geometry that we are going to use later in our proofs. Our approach, however, will be purely analytic, unlike the others, more geometric in nature. Although the first subsection (on the body of flotation) merely reproduces results that are already well known (see \cite{ARSY} and references therein), we decided to include it for the sake of completeness and clarity of presentation. Admittedly, we do not know any reference for the results contained in the second subsection (on the body of illumination). It seems that they are new and original.\\

For the reader's convenience, we will summarize the results in the form of Tables \ref{tab:02} and \ref{tab:04}. We will use the same notation as in the corresponding subsections of the paper. Check marks in the first row indicate the validity of the First Theorem of Dupin.

\begin{remark}
In what follows, we will stop at calculating second-order characteristics of the curves. However, unlike geometric considerations, our method allows almost algorithmic computation of characteristics of any order, if necessary.
\end{remark}

\subsection{The body of flotation}

Let $K\subset\mathbb R^2$ be a convex body with boundary of class $C^2$, let $F_\delta$ be the body of flotation for $K$ and some $\delta\in(0,\Vol_2(K))$, and let $\Gamma_\delta(K)$ be the corresponding centroid curve. Let $\gamma:I\to\mathbb R^2$ be a regular parametrization of $\partial K$, and let $t=t(s)$ be any function of class $C^1$. The area $A=A(s)$ swept out by a point travelling from $\gamma(s)$ to $\gamma(t)$ along the curve $\gamma$ is equal to
$$A=\frac{1}{2}\int_s^t\det(\gamma(u)-\gamma(s),\gamma'(u))\;\mathrm du,$$
whence
\begin{align*}
\frac{\mathrm dA}{\mathrm ds}&=\frac{1}{2}\left(\det(\gamma(t)-\gamma(s),\gamma'(t))\frac{\mathrm dt}{\mathrm ds}-\int_s^t\det(\gamma'(s),\gamma'(u))\;\mathrm du\right)\\
&=\frac{1}{2}\left(\det(\gamma(t)-\gamma(s),\gamma'(t))\frac{\mathrm dt}{\mathrm ds}-\det(\gamma'(s),\gamma(t)-\gamma(s))\right)\\
&=\frac{1}{2}\left(\det(\gamma(t)-\gamma(s),\gamma'(t))\frac{\mathrm dt}{\mathrm ds}+\det(\gamma(t)-\gamma(s),\gamma'(s))\right),
\end{align*}
and finally
\begin{equation}\label{eq:27}\frac{\mathrm dA}{\mathrm ds}=\det\left(\gamma(t)-\gamma(s),\frac{\mathrm d}{\mathrm ds}\left[\frac{\gamma(t)+\gamma(s)}{2}\right]\right).\end{equation}

\begin{figure}
\includegraphics{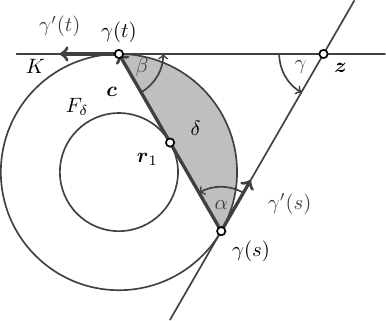}
\caption{Notations used in the argument for the body of flotation}
\label{fig:01}
\end{figure}

\subsubsection{The boundary}

From now on, let $t$ be chosen so that the area $A\equiv\delta$ is constant. Denote $\boldsymbol x(s)\colonequals\gamma(s)$, $\boldsymbol y(s)\colonequals\gamma(t(s))$, $\boldsymbol c\colonequals\boldsymbol y-\boldsymbol x$, $\boldsymbol r_1\colonequals\frac{1}{2}(\boldsymbol x+\boldsymbol y)$. Then \eqref{eq:27} reads
\begin{equation}\label{eq:21}\det(\boldsymbol c,\dot{\boldsymbol r}_1)=0.\end{equation}
It follows directly from the definition that $\boldsymbol r_1$ is a (not necessarily regular) parametrization of $\Pi_\delta$. Indeed, for every $s\in I$, the curve $\boldsymbol r_1$ is tangent to the chord $\boldsymbol c$ that cuts off from $K$ a volume $\delta$ at its midpoint (cf. \thref{thm:19}). From \eqref{eq:21} we may calculate
\begin{equation}\label{eq:14}\frac{\mathrm dt}{\mathrm ds}=-\frac{\det(\boldsymbol c,\gamma'(s))}{\det(\boldsymbol c,\gamma'(t))}.\end{equation}
Note that the denominator does not vanish, for a chord of a convex curve can not be tangent to the curve itself. Combining the above results we get
\begin{equation}\label{eq:33}\dot{\boldsymbol r}_1(s)=\boldsymbol c\cdot\frac{\det(\gamma'(s),\gamma'(t))}{2\det(\boldsymbol c,\gamma'(t))}.\end{equation}
Thus, the parametrization is regular, possibly except for vertex singularities. Now, suppose that $s$ is a regular point of $\boldsymbol r_1$. By \eqref{eq:19}, the oriented curvature of $\Pi_\delta$ at $\boldsymbol r_1$ equals
\begin{equation}\label{eq:20}\kappa_1=\frac{\det(\dot{\boldsymbol r}_1,\ddot{\boldsymbol r}_1)}{|\dot{\boldsymbol r}_1|^3}.\end{equation}
Substituting \eqref{eq:14} to \eqref{eq:20}, after simplification, yields
$$\kappa_1(s)=-\frac{4\det(\boldsymbol c,\gamma'(s))\det(\boldsymbol c,\gamma'(t))}{|\boldsymbol c|^3\det(\gamma'(s),\gamma'(t))}=-\frac{4\cdot(-|\boldsymbol c||\gamma'(s)|\sin\alpha)\cdot(|\boldsymbol c||\gamma'(t)|\sin\beta)}{|\boldsymbol c|^3\cdot(|\gamma'(s)||\gamma'(t)|\sin\gamma)}=\frac{4}{|\boldsymbol c|(\cot\alpha+\cot\beta)}$$
(cf. \fref{fig:01}). Equivalently, the above result can be formulated as
\begin{equation}\label{eq:15}\kappa_1(s)=\frac{\|\boldsymbol c\|^3}{|\boldsymbol c|^3}\end{equation}
(cf. \cite[Lemma~4]{Schutt1994}).

\subsubsection{The centroid curve}

A natural (\emph{a priori} not necessarily regular) parametrization $\boldsymbol r_2$ of $\Gamma_\delta(K)$ is provided by the very definition. Namely,
$$\boldsymbol r_2(s)\colonequals\gamma(s)+\frac{1}{3\delta}\int_s^t(\gamma(u)-\gamma(s))\det(\gamma(u)-\gamma(s),\gamma'(u))\;\mathrm du.$$
Differentiation of the above formula yields
$$\dot{\boldsymbol r}_2(s)=\begin{aligned}[t]&\gamma'(s)+\frac{1}{3\delta}\bigg((\gamma(t)-\gamma(s))\det(\gamma(t)-\gamma(s),\gamma'(t))\frac{\mathrm dt}{\mathrm ds}\\
&\quad-\int_s^t\gamma'(s)\det(\gamma(u)-\gamma(s),\gamma'(u))+(\gamma(u)-\gamma(s))\det(\gamma'(s),\gamma'(u))\;\mathrm du\bigg).\end{aligned}$$
Now, let $C$ be a positively oriented, piecewise smooth, simple closed curve bounding the region $D$ of volume $\delta$ cut off from $K$ by the chord $\boldsymbol c$. Denote $\boldsymbol v(u)\colonequals\gamma(u)-\gamma(s)$, and observe that the latter integral equals
$$\left(\oint_C-\int_{[\boldsymbol y,\boldsymbol x]}\right)\gamma'(s)\det(\boldsymbol v,\mathrm d\boldsymbol v)+\boldsymbol v\det(\gamma'(s),\mathrm d\boldsymbol v).$$
Using Green's formula, we obtain
\begin{align*}
&\oint_C\gamma'(s)\det(\boldsymbol v,\mathrm d\boldsymbol v)+\boldsymbol v\det(\gamma'(s),\mathrm d\boldsymbol v)=\oint_C\begin{pmatrix}\gamma_1'(s)(\boldsymbol v_1\mathrm d\boldsymbol v_2-\boldsymbol v_2\mathrm d\boldsymbol v_1)+\boldsymbol v_1(\gamma_1'(s)\mathrm d\boldsymbol v_2-\gamma_2'(s)\mathrm d\boldsymbol v_1)\\\gamma_2'(s)(\boldsymbol v_1\mathrm d\boldsymbol v_2-\boldsymbol v_2\mathrm d\boldsymbol v_1)+\boldsymbol v_2(\gamma_1'(s)\mathrm d\boldsymbol v_2-\gamma_2'(s)\mathrm d\boldsymbol v_1)\end{pmatrix}\\
&\quad=\oint_C\begin{pmatrix}-\gamma_1'(s)\boldsymbol v_2-\boldsymbol v_1\gamma_2'(s)\\-\gamma_2'(s)\boldsymbol v_2-\boldsymbol v_2\gamma_2'(s)\end{pmatrix}\mathrm d\boldsymbol v_1+\begin{pmatrix}\gamma_1'(s)\boldsymbol v_1+\boldsymbol v_1\gamma_1'(s)\\\gamma_2'(s)\boldsymbol v_1+\boldsymbol v_2\gamma_1'(s)\end{pmatrix}\mathrm d\boldsymbol v_2=\iint_D\begin{pmatrix}2\gamma_1'(s)\\\gamma_2'(s)\end{pmatrix}-\begin{pmatrix}-\gamma_1'(s)\\-2\gamma_2'(s)\end{pmatrix}=\gamma'(s)\cdot 3\delta,
\end{align*}
whereas
$$\int_{[\boldsymbol y,\boldsymbol x]}\gamma'(s)\det(\boldsymbol v,\mathrm d\boldsymbol v)+\boldsymbol v\det(\gamma'(s),\mathrm d\boldsymbol v)=\int_0^1-\boldsymbol c\cdot u\det(\gamma'(s),-\boldsymbol c)\;\mathrm du=\boldsymbol c\cdot\left(-\frac{1}{2}\det(\boldsymbol c,\gamma'(s))\right).$$
Combining the above results, we get
\begin{equation}\label{eq:10}\dot{\boldsymbol r}_2(s)=\boldsymbol c\cdot\left(-\frac{1}{6\delta}\det(\boldsymbol c,\gamma'(s))\right).\end{equation}
Thus, \emph{a posteriori}, the parametrization is regular. Moreover, note that the tangent vector $\dot{\boldsymbol r}_2$ is parallel to the chord $\boldsymbol c$ (cf. \thref{Dup:01}). By \eqref{eq:19}, the oriented curvature of $\Gamma_\delta(K)$ at $\boldsymbol r_2$ equals
\begin{equation}\label{eq:01}\kappa_2=\frac{\det(\dot{\boldsymbol r}_2,\ddot{\boldsymbol r}_2)}{|\dot{\boldsymbol r}_2|^3}.\end{equation}
Substituting \eqref{eq:14} to \eqref{eq:01}, after simplification, yields
\begin{equation}\label{eq:16}\kappa_2(s)=\frac{12\delta}{|\boldsymbol c|^3}\end{equation}
(cf. \eqref{eq:29}).

\begin{table}
\caption{Summary of the results for the body of flotation (see \fref{fig:01} for notation)}
\label{tab:02}
\begin{tblr}{
  colspec = {Q[m,c,.5in]X[m,c]X[m,c]},
  row{1-2} = {.33in}, row{3-4} = {.50in}, row{5-Z} = {.66in},
  cell{3}{1} = {r=2}{}, cell{3}{3} = {r=2}{},
}
\toprule
& $\Pi_\delta(K)$ & $\Gamma_\delta(K)$ \\
\midrule
$\boldsymbol r'\overset{?}{\parallel}\boldsymbol c$ & $\checkmark$ & $\checkmark$ \\
\midrule
$\kappa$ & $\displaystyle\frac{\|\boldsymbol c\|^3}{|\boldsymbol c|^3}$ & $\displaystyle\frac{12\delta}{|\boldsymbol c|^3}$ \\
& $\displaystyle\frac{4}{|\boldsymbol c|(\cot\alpha+\cot\beta)}$ & \\
\midrule
$\kappa'$ & $\displaystyle\frac{24(\cot\alpha-\cot\beta)}{(\cot\alpha+\cot\beta)^2|\boldsymbol c|^2}-\frac{8\left(\frac{\sin^3\alpha}{\kappa(s)}-\frac{\sin^3\beta}{\kappa(t)}\right)}{(\cot\alpha+\cot\beta)^3|\boldsymbol c|}$ & $\displaystyle\frac{216\delta^2(\cot\alpha-\cot\beta)}{|\boldsymbol c|^6}$ \\
\bottomrule
\end{tblr}
\end{table}

\subsection{The body of illumination}

This time, let $K\subset\mathbb R^2$ be a strongly convex body with boundary of class $C^3$, let $I^\delta$ be the body of illumination for $K$ and some $\delta\in(0,\Vol_2(K))$, and let $\Gamma^\delta(K)$ be the corresponding centroid curve. Again, let $\gamma:I\to\mathbb R^2$ be a regular parametrization of $\partial K$, and let $t=t(s)$ be any function of class $C^1$. Denote by $\boldsymbol r_3(s)$ the point of intersection of lines tangent to $\partial K$ at $\gamma(s)$ and $\gamma(t)$. It can be written in a form $\boldsymbol r_3=\gamma+\gamma'\cdot\lambda$, where the coefficient $\lambda=\lambda(s)$ can be determined from the equation $\det(\gamma(t)-\boldsymbol r_3(s),\gamma'(t))=0$. This gives us
$$\boldsymbol r_3(s)=\gamma(s)+\gamma'(s)\cdot\frac{\det(\boldsymbol c,\gamma'(t))}{\det(\gamma'(s),\gamma'(t))}.$$
Note that the denominator does not vanish, since the tangent lines were assumed to intersect. The area $A=A(s)$ swept out by a point travelling from $\gamma(s)$ to $\gamma(t)$ along the curve $\gamma$ is equal to
$$A=-\frac{1}{2}\int_s^t\det(\gamma(u)-\boldsymbol r_3(s),\gamma'(u))\;\mathrm du,$$
whence
\begin{equation}\label{eq:30}\frac{\mathrm dA}{\mathrm ds}=\frac{1}{2}\int_s^t\det(\dot{\boldsymbol r}_3(s),\gamma'(u))\;\mathrm du=-\frac{1}{2}\det(\boldsymbol c,\dot{\boldsymbol r}_3(s)).\end{equation}

\begin{figure}
\includegraphics{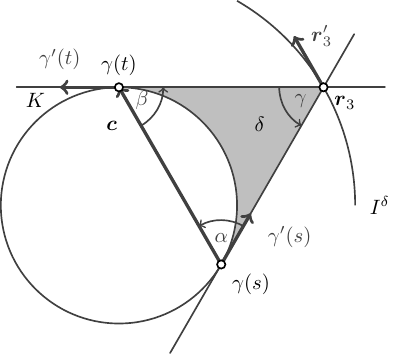}
\caption{Notations used in the argument for the body of illumination}
\label{fig:03}
\end{figure}

\subsubsection{The boundary}

From now on, let $t$ be chosen so that the area $A\equiv\delta$ is constant, which means that $\boldsymbol r_3$ is an (\emph{a priori} not necessarily regular) parametrization of $\Pi^\delta$. Then \eqref{eq:30} reads
\begin{equation}\label{eq:02}\det(\boldsymbol c,\dot{\boldsymbol r}_3)=0.\end{equation}
In particular, it follows that the tangent vector $\dot{\boldsymbol r}_3$ is parallel to the chord $\boldsymbol c$. From \eqref{eq:02} we may calculate
\begin{equation}\label{eq:03}\frac{\mathrm dt}{\mathrm ds}=\frac{\det(\boldsymbol c,\gamma'(t))^2\det(\gamma'(s),\gamma''(s))}{\det(\boldsymbol c,\gamma'(s))^2\det(\gamma'(t),\gamma''(t))}.\end{equation}
Note that the denominator does not vanish, for $\partial K$ is assumed to be strongly convex. Combining the above results, we get
$$\dot{\boldsymbol r}_3(s)=\boldsymbol c\cdot\left(-\frac{\det(\boldsymbol c,\gamma'(t))\det(\gamma'(s),\gamma''(s))}{\det(\boldsymbol c,\gamma'(s))\det(\gamma'(s),\gamma'(t))}\right).$$
Thus, \emph{a posteriori}, the parametrization is regular. Since $\delta$ is finite, the denominator does not vanish. By \eqref{eq:19}, the oriented curvature of $\Pi^\delta$ at $\boldsymbol r_3$ equals
\begin{equation}\label{eq:04}\kappa_3=\frac{\det(\dot{\boldsymbol r}_3,\ddot{\boldsymbol r}_3)}{|\dot{\boldsymbol r}_3|^3}.\end{equation}
Substituting \eqref{eq:03} to \eqref{eq:04}, after simplification, yields
$$\kappa_3(s)=-\frac{\det(\gamma'(s),\gamma'(t))(\det(\boldsymbol c,\gamma'(t))^3\det(\gamma'(s),\gamma''(s))-\det(\boldsymbol c,\gamma'(s))^3\det(\gamma'(t),\gamma''(t)))}{|\boldsymbol c|^3\det(\boldsymbol c,\gamma'(s))\det(\boldsymbol c,\gamma'(t))\det(\gamma'(s),\gamma''(s))\det(\gamma'(t),\gamma''(t))}.$$
Equivalently, the above result can be formulated as
\begin{equation}\label{eq:17}\kappa_3(s)=\frac{4\left(\frac{\sin^3\alpha}{\kappa(s)}+\frac{\sin^3\beta}{\kappa(t)}\right)}{\|\boldsymbol c\|^3}\end{equation}
(cf. \fref{fig:03}).

\subsubsection{The centroid curve}

Again, a natural (\emph{a priori} not necessarily regular) parametrization $\boldsymbol r_4$ of $\Gamma^\delta(K)$ is provided by the very definition. Namely,
$$\boldsymbol r_4(s)\colonequals\boldsymbol r_3(s)-\frac{1}{3\delta}\int_s^t(\gamma(u)-\boldsymbol r_3(s))\det(\gamma(u)-\boldsymbol r_3(s),\gamma'(u))\;\mathrm du.$$
Differentiation of the above formula yields
$$\dot{\boldsymbol r}_4(s)\colonequals\dot{\boldsymbol r}_3(s)+\frac{1}{3\delta}\int_s^t\dot{\boldsymbol r}_3(s)\det(\gamma(u)-\boldsymbol r_3(s),\gamma'(u))+(\gamma(u)-\boldsymbol r_3(s))\det(\dot{\boldsymbol r}_3(s),\gamma'(u))\;\mathrm du.$$
Now, let $C$ be a positively oriented, piecewise smooth, simple closed curve bounding the region $D$ of volume $\delta$, being the silhouette cone determined by the chord $\boldsymbol c$. Denote $\boldsymbol v(u)\colonequals\gamma(u)-\boldsymbol r_3(s)$, and observe that the latter integral equals
$$-\left(\oint_C-\int_{[\boldsymbol r_3,\boldsymbol y]}-\int_{[\boldsymbol x,\boldsymbol r_3]}\right)\dot{\boldsymbol r}_3(s)\det(\boldsymbol v,\mathrm d\boldsymbol v)+\boldsymbol v\det(\dot{\boldsymbol r}_3(s),\mathrm d\boldsymbol v).$$
Again, using Green's formula, we obtain
$$\oint_C\dot{\boldsymbol r}_3(s)\det(\boldsymbol v,\mathrm d\boldsymbol v)+\boldsymbol v\det(\dot{\boldsymbol r}_3(s),\mathrm d\boldsymbol v)=\dot{\boldsymbol r}_3(s)\cdot 3\delta,$$
whereas
$$\int_{[\boldsymbol r_3,\boldsymbol y]}\dot{\boldsymbol r}_3(s)\det(\boldsymbol v,\mathrm d\boldsymbol v)+\boldsymbol v\det(\dot{\boldsymbol r}_3(s),\mathrm d\boldsymbol v)=-\frac{1}{2}(\boldsymbol y-\boldsymbol r_3)\det(\boldsymbol y-\boldsymbol r_3,\dot{\boldsymbol r}_3(s))$$
and
$$\int_{[\boldsymbol x,\boldsymbol r_3]}\dot{\boldsymbol r}_3(s)\det(\boldsymbol v,\mathrm d\boldsymbol v)+\boldsymbol v\det(\dot{\boldsymbol r}_3(s),\mathrm d\boldsymbol v)=+\frac{1}{2}(\boldsymbol x-\boldsymbol r_3)\det(\boldsymbol x-\boldsymbol r_3,\dot{\boldsymbol r}_3(s)).$$
Combining the above results, we get
$$\dot{\boldsymbol r}_4(s)=\boldsymbol c\cdot\frac{1}{6\delta}\frac{\det(\boldsymbol c,\gamma'(t))^2\det(\gamma'(s),\gamma''(s))}{\det(\gamma'(s),\gamma'(t))^2}.$$
Note that the denominator does not vanish, since the tangent lines were assumed to intersect. Thus, \emph{a posteriori} the parametrization is regular. Moreover, note that the tangent vector $\dot{\boldsymbol r}_4$ is parallel to the chord $\boldsymbol c$. By \eqref{eq:19}, the oriented curvature of $\Gamma^\delta(K)$ at $\boldsymbol r_4$ equals
\begin{equation}\label{eq:05}\kappa_4=\frac{\det(\dot{\boldsymbol r}_4,\ddot{\boldsymbol r}_4)}{|\dot{\boldsymbol r}_4|^3}.\end{equation}
Substituting \eqref{eq:03} to \eqref{eq:05}, after simplification, yields
$$\kappa_4(s)=\frac{6\delta\det(\gamma'(s),\gamma'(t))^2(\det(\boldsymbol c,\gamma'(t))^3\det(\gamma'(s),\gamma''(s))-\det(\boldsymbol c,\gamma'(s))^3\det(\gamma'(t),\gamma''(t)))}{|\boldsymbol c|^3\det(\boldsymbol c,\gamma'(s))^2\det(\boldsymbol c,\gamma'(t))^2\det(\gamma'(s),\gamma''(s))\det(\gamma'(t),\gamma''(t))}.$$
Equivalently, the above result can be formulated as
\begin{equation}\label{eq:18}\kappa_4(s)=\frac{96\delta\left(\frac{\sin^3\alpha}{\kappa(s)}+\frac{\sin^3\beta}{\kappa(t)}\right)}{\|\boldsymbol c\|^6}\end{equation}
(cf. \fref{fig:03}).

\begin{table}
\caption{Summary of the results for the body of illumination (see \fref{fig:03} for notation)}
\label{tab:04}
\begin{tblr}{
  colspec = {Q[m,c,.5in]X[m,c]X[m,c]},
  row{1-2} = {.33in}, row{3-Z} = {.66in},
}
\toprule
& $\Pi^\delta(K)$ & $\Gamma^\delta(K)$ \\
\midrule
$\boldsymbol r'\overset{?}{\parallel}\boldsymbol c$ & $\checkmark$ & $\checkmark$ \\
\midrule
$\kappa$ & $\displaystyle\frac{4\left(\frac{\sin^3\alpha}{\kappa(s)}+\frac{\sin^3\beta}{\kappa(t)}\right)}{\|\boldsymbol c\|^3}$ & $\displaystyle\frac{96\delta\left(\frac{\sin^3\alpha}{\kappa(s)}+\frac{\sin^3\beta}{\kappa(t)}\right)}{\|\boldsymbol c\|^6}$ \\
\bottomrule
\end{tblr}
\end{table}

\subsection{Limiting cases for centrally symmetric body}\label{sec:07}

For the purposes of this section, let us assume that $K\subset\mathbb R^n$ is origin-symmetric. We will investigate the asymptotic shape of the body of flotation and the body of illumination, as the parameter $\delta$ approaches the boundary of its domain. Clearly, we have
$$\lim_{\delta\to 0^+}d_{\mathrm H}(F_\delta(K),K)=0$$
(cf. \cite[\S 4.1]{Ryabogin2023}), and similarly
$$\lim_{\delta\to 0^+}d_{\mathrm H}(I^\delta(K),K)=0,$$
where $d_{\mathrm H}$ denotes the Hausdorff distance between convex bodies \cite[\S 1.8]{Schneider_2013}. Indeed, the support functions $h_{F_\delta(K)}$ and $h_{I^\delta(K)}$ are monotonically convergent to the support function $h_K$, so by Dini's theorem the convergence on $\mathbb S^{n-1}$ is uniform, and thus we have the convergence of convex bodies (cf. \cite[Theorem~1.8.11]{Schneider_2013}).\\

\begin{figure}
\includegraphics{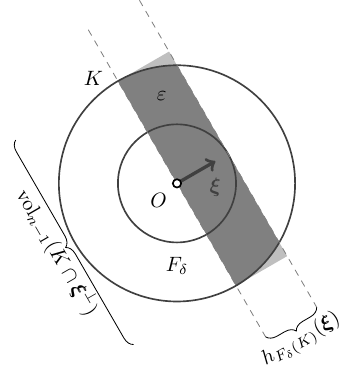}
\caption{Notations used in the argument for the asymptotic body of flotation}
\label{fig:02}
\end{figure}

\begin{notation}
In what follows, we will use the little-o notation. That is, $f(x)=o(g(x))$ if $$\lim_{x\to\infty}\frac{f(x)}{g(x)}=0.$$
\end{notation}

It is not difficult to see that
$$\lim_{\varepsilon\to 0^+}d_{\mathrm H}(\varepsilon^{-1}F_\delta(K),\mathcal I^*K)=0,$$
where $\delta\colonequals\frac{1}{2}\Vol_n(K)-\varepsilon$, and $\mathcal I^*K\colonequals(\mathcal IK)^*$ is the polar of the intersection body of $K$ \cite[\S 8.1]{Gardner_2006}. Indeed, for every $\boldsymbol\xi\in\mathbb S^{n-1}$ we have
$$\Vol_{n-1}(K\cap\boldsymbol\xi^\perp)\cdot h_{F_\delta(K)}(\boldsymbol\xi)=\varepsilon+o(\varepsilon),$$
(cf. \fref{fig:02}) whence
$$h_{\varepsilon^{-1}F_\delta(K)}(\boldsymbol\xi)=\varepsilon^{-1}h_{F_\delta(K)}(\boldsymbol\xi)=\frac{1}{\Vol_{n-1}(K\cap\boldsymbol\xi^\perp)}+o(1)=\frac{1}{\rho_{IK}(\boldsymbol\xi)}+o(1)=h_{\mathcal I^*K}(\boldsymbol\xi)+o(1).$$
Thus, the sequence of support functions converges pointwise, and the conclusion follows from \cite[Theorem~1.8.12]{Schneider_2013}. Note that since $K$ is a symmetric convex body, $\mathcal IK$ is also a symmetric convex body \cite[Corollary~8.1.11]{Gardner_2006}, so the polar body $\mathcal I^*K$ is well-defined.\\

\begin{figure}
\includegraphics{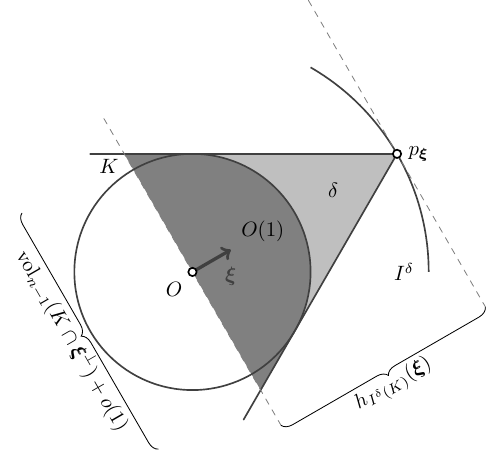}
\caption{Notations used in the argument for the asymptotic body of illumination}
\label{fig:05}
\end{figure}

Similarly, we have
$$\lim_{\delta\to\infty}d_{\mathrm H}((n\delta)^{-1}I^\delta(K),\mathcal I^*K)=0.$$
Indeed, fix any $\boldsymbol\xi\in\mathbb S^{n-1}$ and let $p_{\boldsymbol\xi}\in\Pi^\delta(K)$ be the point on the boundary such that $K\subset H_{\boldsymbol\xi}^-+p_{\boldsymbol\xi}$. The volume of the conical hull of $K$ with respect to the vertex $p_{\boldsymbol\xi}$ intersected with $H_{\boldsymbol\xi}^+$ can be computed in two ways:
$$\frac{1}{n}(\Vol_{n-1}(K\cap\boldsymbol\xi^\perp)+o(1))\cdot h_{I^\delta(K)}(\boldsymbol\xi)=\delta+O(1).$$
On the left-hand side, we compute it as a volume of a cone with base being essentially $K\cap\boldsymbol\xi^\perp$ and height equal to $h_{I^\delta(K)}(\boldsymbol\xi)$. On the right-hand side, we compute it as essentially the volume of the volumetric light cone emanating from $p_{\boldsymbol\xi}$ and the volume of $K$ intersected with $H_{\boldsymbol\xi}^+$ (cf. \fref{fig:05}). Hence, we obtain
$$h_{(n\delta)^{-1}I^\delta(K)}(\boldsymbol\xi)=(n\delta)^{-1}h_{I^\delta(K)}(\boldsymbol\xi)=\frac{1+o(1)}{\Vol_{n-1}(K\cap\boldsymbol\xi^\perp)+o(1)}=\frac{1}{\rho_{IK}(\boldsymbol\xi)}+o(1)=h_{\mathcal I^*K}(\boldsymbol\xi)+o(1).$$
Again, the sequence of support functions converges pointwise, and the conclusion follows from \cite[Theorem~1.8.12]{Schneider_2013}.\\

Finally, observe that if the boundary of $K$ is given by a parametric curve, the boundary of $\mathcal I^*K$ also admits a simple parametrization:

\begin{proposition}
Let $K\subset\mathbb R^2$ be an origin-symmetric convex body with boundary of class $C^2$, and let $\gamma:I\to\mathbb R^2$ be a regular parametrization of its boundary $\partial K$. Then the curve
$$\gamma_*\colonequals\frac{\gamma'}{2\det(\gamma,\gamma')}$$
is a regular parametrization of the boundary of $\mathcal I^*K$.
\end{proposition}

\begin{proof}
Since $\gamma$ is a closed curve, the same holds for $\gamma_*$. It is also easy to see that $\gamma_*$ is simple, so it is a Jordan curve, and thus it encloses a compact set $K_*$. Observe that $\gamma_*'\parallel\gamma$. Indeed,
\begin{equation}\label{eq:26}\gamma_*'=\frac{\gamma''\det(\gamma,\gamma')-\gamma'\det(\gamma,\gamma'')}{2\det(\gamma,\gamma')^2}=\gamma\cdot\left(-\frac{\det(\gamma',\gamma'')}{2\det(\gamma,\gamma')^2}\right).\end{equation}
In particular, the Gauss map of $K_*$ is injective, and thus $K_*$ is convex. Now, fix any $\boldsymbol\xi\in\mathbb S^1$, and let $s_{\boldsymbol\xi}\in I$ be a point such that $\gamma(s_{\boldsymbol\xi})\perp\boldsymbol\xi$. We have
$$\rho_{\mathcal IK}(\boldsymbol\xi)\cdot h_{K_*}(\boldsymbol\xi)=\det(2\gamma(s_{\boldsymbol\xi}),\gamma_*(s_{\boldsymbol\xi}))=1,$$
which concludes the proof.
\end{proof}

\subsection{Relations between the body of flotation and the body of buoyancy}

In this section, we will investigate two remarkable properties of the body of flotation and the body of buoyancy on the plane. We believe them to be rather accidental phenomena, however, and most likely they occur neither in higher dimensions nor for other types of bodies introduced in \sref{sec:04}.

\begin{notation}
Before delving into formulations and proofs, we will introduce some notation that will be very useful. It is $\frac{3}{2}\delta$ rather than $\delta$ that appears in most formulas. We therefore define $\dbar\colonequals\frac{3}{2}\delta$.
\end{notation}

Firstly, we will show that the affine normal to the centroid curve is always parallel to the affine bisector of the corresponding chord of flotation $[\boldsymbol x,\boldsymbol y]$, defined as the median of the projective triangle $\triangle\boldsymbol{xzy}$ from the vertex $\boldsymbol z$ (cf. \fref{fig:01}).

\begin{proposition}\thlabel{thm:03}
Let $K\subset\mathbb R^2$ be a convex body with boundary of class $C^2$, and let $\delta\in(0,\Vol_2(K))$. Then the affine normal to the centroid curve is always parallel to the affine bisector of the corresponding chord $[\boldsymbol x,\boldsymbol y]$. Moreover, if the corresponding directions $\dot{\boldsymbol x},\dot{\boldsymbol y}$ are not parallel, the following equality holds:
$$\boldsymbol r_2''=\frac{8\dbar^{\frac{1}{3}}}{\|\boldsymbol c\|^3}(\boldsymbol r_1-\boldsymbol z).$$
\end{proposition}

\begin{proof}
We have
$$\boldsymbol r_2'=\dot{\boldsymbol r}_2\frac{\mathrm ds}{\mathrm d\sigma},\quad\boldsymbol r_2''=\ddot{\boldsymbol r}_2\left(\frac{\mathrm ds}{\mathrm d\sigma}\right)^2+\dot{\boldsymbol r}_2\frac{\mathrm d^2s}{\mathrm d\sigma^2}$$
(cf. \cite[p.~7]{Buchin}) and
$$\frac{\mathrm ds}{\mathrm d\sigma}=\det(\dot{\boldsymbol r}_2,\ddot{\boldsymbol r}_2)^{-\frac{1}{3}},\quad\frac{\mathrm d^2s}{\mathrm d\sigma^2}=-\frac{1}{3}\det(\dot{\boldsymbol r}_2,\ddot{\boldsymbol r}_2)^{-\frac{5}{3}}\det(\dot{\boldsymbol r}_2,\dddot{\boldsymbol r}_2),$$
(cf. \cite[p.~8]{Buchin}) whence
$$\boldsymbol r_2''=\ddot{\boldsymbol r}_2\det(\dot{\boldsymbol r}_2,\ddot{\boldsymbol r}_2)^{-\frac{2}{3}}-\frac{1}{3}\dot{\boldsymbol r}_2\det(\dot{\boldsymbol r}_2,\ddot{\boldsymbol r}_2)^{-\frac{5}{3}}\det(\dot{\boldsymbol r}_2,\dddot{\boldsymbol r}_2)$$
(cf. \cite[(I.3.20)]{Buchin}). Keeping in mind formulas \eqref{eq:10} and \eqref{eq:14}, the conclusion follows from a tedious but straightforward computation.
\end{proof}

Unfortunately, we were unable to find such a simple geometric definition of the affine normal, neither for the body of flotation nor for the bodies of illumination.\\

Finally, we can also deduce one more, rather unanticipated equality. The \emph{affine hypersurface area} of a convex body $K\subset\mathbb R^n$ with sufficiently smooth boundary is given by
$$\Omega(\partial K)=\int_{\partial K}\kappa(x)^{\frac{1}{n+1}}\;\mathrm d\mathcal H^{n-1}(x),$$
where $\kappa(x)$ is the generalized curvature of $\partial K$ at $x$, and $\mathcal H^{n-1}$ denotes the hypersurface measure on $\partial K$. C. Sch\"utt and E.M. Werner \cite{SchuttWerner} showed that for all convex bodies $K\in\mathbb R^n$ we have
\begin{equation}\label{eq:08}\lim_{\delta\to 0^+}\frac{\Vol_n(K)-\Vol_n(F_\delta(K))}{\delta^{\frac{2}{n+1}}}=\frac{1}{2}\left(\frac{n+1}{\Vol_{n-1}(\mathbb B^{n-1})}\right)^{\frac{2}{n+1}}\Omega(\partial K).\end{equation}
On a plane, the affine hypersurface area coincides with the affine arc length \eqref{eq:07}. Moreover, it turns out that the expression on the left-hand side of \eqref{eq:08} for $n=2$ is actually equal to the affine arc length of $\Gamma_\delta(K)$.

\begin{proposition}
Let $K\subset\mathbb R^2$ be a convex body with boundary of class $C^1$, and let $\delta\in(0,\Vol_2(K))$. Then the following equality holds:
$$\frac{\Vol_2(K)-\Vol_2(F_\delta(K))}{\dbar^{\frac{2}{3}}}=\frac{1}{2}\Omega(\Gamma_\delta(K)).$$
\end{proposition}

\begin{proof}
By the definition of the affine arc length, we have
$$\Omega(\Gamma_\delta(K))=\int_I\kappa_2(s)^{\frac{1}{3}}|\dot{\boldsymbol r}_2(s)|\;\mathrm ds=\int_I\frac{(12\delta)^{\frac{1}{3}}}{|\boldsymbol c|}|\boldsymbol c|\frac{1}{6\delta}\det(-\boldsymbol c,\gamma'(s))\;\mathrm ds=\frac{1}{2\dbar^{\frac{2}{3}}}\int_I\det(-\boldsymbol c,\gamma'(s))\;\mathrm ds.$$
On the other hand, we have
\begin{align}
\nonumber\Vol_2(K)-\Vol_2(F_\delta(K))&=\frac{1}{2}\int_I\det(\gamma(s),\gamma'(s))\;\mathrm ds-\frac{1}{2}\int_I\det(\boldsymbol r_1(s),\dot{\boldsymbol r}_1(s))\;\mathrm ds\\
\label{eq:06}&=\frac{1}{2}\int_I\det(\boldsymbol x,\dot{\boldsymbol x})\;\mathrm ds-\frac{1}{8}\int_I\det(\boldsymbol x+\boldsymbol y,\dot{\boldsymbol x}+\dot{\boldsymbol y})\;\mathrm ds.
\end{align}
Now, recall that
$$\int_I\det(\boldsymbol x,\dot{\boldsymbol x})\;\mathrm ds=2\Vol_2(K)=\int_I\det(\boldsymbol y,\dot{\boldsymbol y})\;\mathrm ds$$
(cf. \cite[\S 1-7, (1)]{do2016differential}) and
$$\det(\boldsymbol y-\boldsymbol x,\dot{\boldsymbol x}+\dot{\boldsymbol y})=0,$$
whence
$$\int_I\det(\boldsymbol x,\dot{\boldsymbol y})\;\mathrm ds=\int_I\det(\boldsymbol y,\dot{\boldsymbol x})\;\mathrm ds.$$
Simplifying \eqref{eq:06} yields
$$\Vol_2(K)-\Vol_2(F_\delta(K))=\frac{1}{4}\int_I\det(\boldsymbol x-\boldsymbol y,\dot{\boldsymbol x})\;\mathrm ds=\frac{1}{4}\int_I\det(-\boldsymbol c,\gamma'(s))\;\mathrm ds,$$
which concludes the proof.
\end{proof}

As an immediate corollary, we obtain that the function $\delta\mapsto\Omega(\Gamma_\delta(K))$ is uniformly continuous, which is by no means obvious since the affine hypersurface area as a function on the set $\mathcal K^2$ of convex bodies in $\mathbb R^2$ is not continuous in the Hausdorff metric.

\section{Proofs of the theorems}\label{sec:06}

We can construct the following polarity $\varphi_K$ on the real projective plane with respect to a convex body $K$ with boundary of class $C^1$, which is a bijection between points of the exterior of $K$ and lines intersecting the interior of $K$:

\begin{definition}
In the projective plane, fix a convex body $K$ with boundary of class $C^1$. Let $p$ be a point of the exterior of $K$. Precisely two lines are passing through $p$ that are tangent to $K$. The line $\ell\equalscolon\varphi_K^{-1}(p)$ connecting the points of tangency intersects the interior of $K$, and is called a \emph{polar} of the point $p$ with respect to the body $K$. Conversely, let $\ell$ be a line intersecting the interior of $K$. There are precisely two points where $\ell$ intersects the boundary of $K$. The point $p\equalscolon\varphi_K(\ell)$ being the intersection of the lines tangent to the boundary of $K$ at those points is called a \emph{pole} of the line $\ell$ with respect to the body $K$.
\end{definition}

\begin{remark}
Note that if $K$ is the unit disc, then the above polarity is precisely the reciprocation in the Euclidean plane. However, since in general it does not preserve incidence, it is not a duality in the strict sense.
\end{remark}

\subsection{Affine distance between the endpoints of chords of flotation}

Using the computational results from \sref{sec:05}, we can now establish several theorems related to the homothety conjecture as formulated in \thref{que:03}. The polarity $\varphi_K$ naturally emerges in the context of bodies of flotation and illumination. Indeed, the body of flotation $F_\delta(K)$ is defined in terms of its tangent bundle, which consists of lines intersecting the interior of $K$, whereas the body of illumination $I^\delta(K)$ is defined as a set of points of the exterior of $K$. Although in general the duality between the body of flotation and the body of illumination is only on the level of philosophy, under the assumption of the homothety conjecture, it can actually be formalized:

\begin{theorem}\thlabel{thm:01}
Let $K\subset\mathbb R^2$ be a convex body with boundary of class $C^2$, and let $\delta,\hat\delta>0$. Then:
\begin{enumerate}
\item\label{it:01} $\Pi_\delta(K)$ is homothetic to $\Gamma_\delta(K)$ with ratio $\lambda>1$ if and only if $\|\boldsymbol c\|^3$ is constant and equal to $12\delta\lambda$,
\item\label{it:02} $\Pi^{\hat\delta}(K)$ is homothetic to $\Gamma^{\hat\delta}(K)$ with ratio $\hat\lambda<1$ if and only if $\|\boldsymbol c\|^3$ is constant and equal to $24\hat\delta\hat\lambda$.
\end{enumerate}
\end{theorem}

\begin{proof}
We will prove only part (\ref{it:01}), as the proof of part (\ref{it:02}) is completely analogous. By \eqref{eq:33} and \eqref{eq:10}, we have
\begin{equation}\label{eq:34}\dot{\boldsymbol r}_2(s)=\dot{\boldsymbol r}_1(s)\cdot\left(-\frac{1}{3\delta}\frac{\det(\boldsymbol c,\gamma'(s))\det(\boldsymbol c,\gamma'(t))}{\det(\gamma'(s),\gamma'(t))}\right)=\dot{\boldsymbol r}_1(s)\cdot\frac{\|\boldsymbol c\|^3}{12\delta}.\end{equation}
Now, suppose that $\Pi_\delta(K)$ is homothetic to $\Gamma_\delta(K)$ with ratio $\lambda$. Since $\dot{\boldsymbol r_1},\dot{\boldsymbol r}_2$ are both parallel to $\boldsymbol c$, points $\boldsymbol r_1,\boldsymbol r_2$ must correspond to each other in homothety. Thus we have $\dot{\boldsymbol r}_2=\lambda\dot{\boldsymbol r}_1$ and the desired conclusion follows. Conversely, suppose that $\|\boldsymbol c\|^3$ is constant and equal to $12\delta\lambda$. Then \eqref{eq:34} reads $\dot{\boldsymbol r}_2=\lambda\dot{\boldsymbol r}_1$, whence again the desired conclusion follows. This ends the proof.
\end{proof}

\begin{corollary}
Let $K\subset\mathbb R^2$ be a convex body with boundary of class $C^2$, and let $\delta,\hat\delta>0$ and $\lambda,\hat\lambda>0$ satisfy
\begin{equation}\label{eq:24}1/(\hat\delta\hat\lambda)=2/(\delta\lambda)\quad\text{and}\quad 1/\hat\lambda+2/\lambda=3.\end{equation}
Then the following are equivalent:
\begin{enumerate}
\item\label{it:03} $\Pi_\delta(K)$ is homothetic to $\Gamma_\delta(K)$ with ratio $\lambda$,
\item\label{it:04} $\Pi^{\hat\delta}(K)$ is homothetic to $\Gamma^{\hat\delta}(K)$ with ratio $\hat\lambda$,
\item\label{it:05} $\Pi_\delta(K)$ and $\Pi^{\hat\delta}(K)$ are dual to each other with respect to the polarity $\varphi_K$.
\end{enumerate}
\end{corollary}

\begin{proof}
We will prove only the equivalence of (\ref{it:03}) and (\ref{it:05}), as the proof of the equivalence of (\ref{it:04}) and (\ref{it:05}) is completely analogous. Suppose that $\Pi_\delta(K)$ is homothetic to $\Gamma_\delta(K)$ with ratio $\lambda$. Then, by \thref{thm:01} (\ref{it:01}), $\|\boldsymbol c\|^3=12\delta\lambda$ is constant. Solving \eqref{eq:24} for $\hat\delta,\hat\lambda$ yields $\hat\delta=\frac{3}{2}\delta\lambda-\delta$. Now, observe that $\varphi_K(\Pi_\delta(K))=\Pi^{\hat\delta}(K)$ with the same chords of flotation. Indeed, $\hat\delta=\frac{1}{8}\|\boldsymbol c\|^3-\delta=\Vol_2(\triangle\boldsymbol{xyz})-\delta$ is precisely the volume of the silhouette cone. Conversely, suppose that $\varphi_K(\Pi_\delta(K))=\Pi^{\hat\delta}(K)$. By construction, the chords of flotation are again the same. Since $\|\boldsymbol c\|^3=8\Vol_2(\triangle\boldsymbol{xyz})=8(\delta+\hat\delta)=12\delta\lambda$ is constant, by \thref{thm:01} (\ref{it:01}), $\Pi_\delta(K)$ is homothetic to $\Gamma_\delta(K)$ with ratio $\lambda$. This concludes the proof.
\end{proof}

\subsection{Affine arc length of floating arcs}

It is a classical fact that a complete system for differential invariants for the planar Euclidean group $\mathbb R^2\rtimes\mathrm{O}(\mathbb R^2)$ is provided by the curvature $\kappa$, and its successive derivatives with respect to the arc length $\kappa',\kappa'',\ldots$ \cite[p.~175]{Olver_1999}. So far, we have not gone beyond the simplest second-order invariant in our considerations. By comparing third-order invariants of $\Pi_\delta(K)$ and $\Gamma_\delta(K)$ we obtain the following original theorem:

\begin{theorem}\thlabel{thm:02}
Let $K\subset\mathbb R^2$ be a convex body with boundary of class $C^2$, and let $\delta\in(0,\Vol_2(K))$. If $\Pi_\delta(K)$ is homothetic to $\Gamma_\delta(K)$, then for every chord $\boldsymbol c$ of $K$ tangent to $\Pi_\delta(K)$ we have
\begin{equation}\label{eq:13}\frac{\sin^3\alpha}{\kappa(s)}=\frac{\sin^3\beta}{\kappa(t)}\end{equation}
(cf. \fref{fig:01}). Moreover, \eqref{eq:13} is equivalent to the fact that the affine arc length of the boundary cut off from $K$ by the chord $\boldsymbol c$ is constant.
\end{theorem}

\begin{proof}
Suppose that $\Pi_\delta(K)$ is homothetic to $\Gamma_\delta(K)$ with ratio $\lambda$. Using results from Table~\ref{tab:02} we immediately get
\begin{gather}
\label{eq:11}\frac{4}{|\boldsymbol c|(\cot\alpha+\cot\beta)}=\lambda\frac{12\delta}{|\boldsymbol c|^3},\\
\label{eq:12}\frac{24(\cot\alpha-\cot\beta)}{(\cot\alpha+\cot\beta)^2|\boldsymbol c|^2}-\frac{8\left(\frac{\sin^3\alpha}{\kappa(s)}-\frac{\sin^3\beta}{\kappa(t)}\right)}{(\cot\alpha+\cot\beta)^3|\boldsymbol c|}=\lambda^2\frac{216\delta^2(\cot\alpha-\cot\beta)}{|\boldsymbol c|^6}.
\end{gather}
Computing $\lambda$ from \eqref{eq:11} and substituting it to \eqref{eq:12} yields
$$-\frac{8\left(\frac{\sin^3\alpha}{\kappa(s)}-\frac{\sin^3\beta}{\kappa(t)}\right)}{(\cot\alpha+\cot\beta)^3|\boldsymbol c|}=0,$$
which concludes the first part of the proof.\\

To see the equivalence from the second part, observe that the affine arc length of the boundary cut off from $K$ by the chord $\boldsymbol c$ is given by
$$\int_s^t\det(\gamma'(u),\gamma''(u))^{\frac{1}{3}}\;\mathrm du.$$
The derivative of this integral is expressible as
\begin{align*}
&\det(\gamma'(t),\gamma''(t))^{\frac{1}{3}}\cdot\frac{\mathrm dt}{\mathrm ds}-\det(\gamma'(s),\gamma''(s))^{\frac{1}{3}}\\
&\quad=-\det(\gamma'(t),\gamma''(t))^{\frac{1}{3}}\frac{\det(\boldsymbol c,\gamma'(s))}{\det(\boldsymbol c,\gamma'(t))}-\det(\gamma'(s),\gamma''(s))^{\frac{1}{3}}\\
&\quad=-\det(\gamma'(t),\gamma''(t))^{\frac{1}{3}}\frac{-|\boldsymbol c||\gamma'(s)|\sin\alpha}{|\boldsymbol c||\gamma'(t)|\sin\beta}-\det(\gamma'(s),\gamma''(s))^{\frac{1}{3}}\\
&\quad=|\gamma'(s)|\sin\alpha\left(\frac{\det(\gamma'(t),\gamma''(t))^{\frac{1}{3}}}{|\gamma'(t)|\sin\beta}-\frac{\det(\gamma'(s),\gamma''(s))^{\frac{1}{3}}}{|\gamma'(s)|\sin\alpha}\right)\\
&\quad=|\gamma'(s)|\sin\alpha\left(\frac{\kappa(t)^{\frac{1}{3}}}{\sin\beta}-\frac{\kappa(s)^{\frac{1}{3}}}{\sin\alpha}\right)\\
&\quad=0,
\end{align*}
where the first equality follows from \eqref{eq:14}, and the last equality is an immediate consequence of \eqref{eq:13}. This concludes the proof.
\end{proof}

\begin{remark}
Condition \eqref{eq:13} may be equivalently expressed as
\begin{equation}\label{eq:22}\kappa(s)^{-1}(\boldsymbol n(s)\cdot(\gamma(s)-\gamma(t)))^3=\kappa(t)^{-1}(\boldsymbol n(t)\cdot(\gamma(t)-\gamma(s)))^3.\end{equation}
Note a striking resemblance to the condition
\begin{equation}\label{eq:23}\kappa(s)^{-1}(\boldsymbol n(s)\cdot\gamma(s))^3=\Const,\end{equation}
defining the class of convex bodies $\mathscr F_2$ introduced by C.M. Petty in \cite[\S 6]{Petty1985}. However, observe that \eqref{eq:23} depends on the choice of the origin (actually, the author assumed that it is the Santal\'o point of $K$), whereas \eqref{eq:22} is translation-invariant.
\end{remark}

In the same paper, the author proved that $\mathscr F_2$ consists only of ellipses (see \cite[Lemma~8.1]{Petty1985}). We wish we could do the same also in our case, assuming that \eqref{eq:22} holds merely for chords supporting $\Pi_\delta$, i.e., for $t$ functionally dependent on $s$. Without this stipulation, the theorem follows from a somewhat standard argument in differential equations:

\begin{theorem}
Let $K\subset\mathbb R^2$ be a convex body with boundary of class $C^2$. Suppose that \eqref{eq:13} holds for every chord in some open set $U$ of chords that covers $\partial K$. Then $K$ is an ellipse.
\end{theorem}

\begin{proof}
Fix any point $x\in\partial K$ and any chord $[x,y]\in U$ that passes through $x$. By rotating the coordinate system, without loss of generality, we may assume that the chord is vertical. Hence, the boundary curve $\gamma$ in some neighbourhood of $x=\gamma(s_*)$ can be locally represented as a graph of a function $\gamma(s_*+s)=(s,f_1(s))$, and on some neighbourhood of $y=\gamma(t_*)$ can be locally represented as a graph of a function $\gamma(t_*+t)=(t,f_2(t))$. Then \eqref{eq:13} reads
\begin{equation}\label{eq:25}f_1''(s)=\left(\frac{f_2(t)-f_1(s)-(t-s)f_1'(s)}{f_1(s)-f_2(t)-(s-t)f_2'(t)}\right)^3f_2''(t).\end{equation}
Observe that the denominator does not vanish, as it is equal to $\det(\gamma(t_*+t)-\gamma(s_*+s),\gamma'(t_*+t))$. Since the right-hand side of \eqref{eq:25} is a smooth function of $s,f_1(s),f_1'(s)$, it satisfies the assumptions of Picard-Lindel\"of theorem, whence \eqref{eq:25} has a unique solution $f_1(s)$ on some neighbourhood of $s_*$. On the other hand, any quadric is known to satisfy the assumptions of \thref{thm:02}, and thus the unique quadric having second-order contact with $\gamma$ at $y$ and first-order contact with $\gamma$ at $x$ satisfies \eqref{eq:25}. It follows that $\gamma$ is a quadric on some neighbourhood of $s_*$. Because the choice of $s_*$ was arbitrary, this concludes the proof.
\end{proof}

\begin{corollary}
Let $K\subset\mathbb R^2$ be a convex body with boundary of class $C^2$. The set of all values of $\delta$ for which $\Pi_\delta(K)$ is homothetic to $\Gamma_\delta(K)$ is nowhere dense in $(0,\Vol_2(K))$, unless $K$ is an ellipse.
\end{corollary}

\begin{figure}
\includegraphics{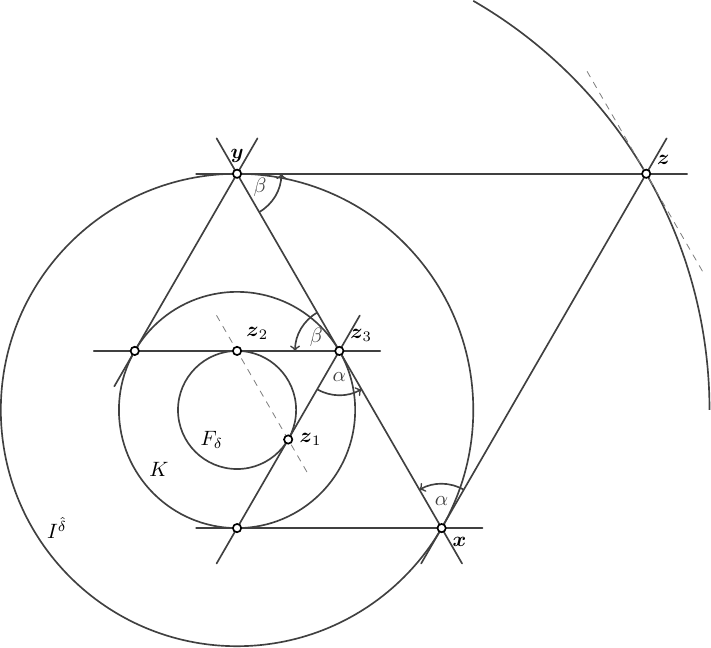}
\caption{Construction of the dual body $\varphi_{I^{\hat\delta}(K)}(K)$}
\label{fig:06}
\end{figure}

\begin{remark}
Under the assumption of \thref{thm:02}, not only the pair $(F_\delta(K),K)$ satisfies \eqref{eq:13}, but the same equality holds also for $(K,I^{\hat\delta}(K))$. Indeed, by \eqref{eq:17} and \thref{thm:02,thm:01} we have
$$\frac{\sin^3\alpha}{\kappa_3(\boldsymbol x)}=\frac{\sin^3\alpha}{\frac{4\left(\frac{\sin^3\alpha}{\kappa_1(\boldsymbol z_3)}+\frac{\sin^3\alpha}{\kappa_1(\boldsymbol z_3)}\right)}{\|\boldsymbol c_1\|^3}}=\frac{\kappa_1(\boldsymbol z_3)\|\boldsymbol c_1\|^3}{8}=\frac{\kappa_1(\boldsymbol z_3)\|\boldsymbol c_2\|^3}{8}=\frac{\sin^3\beta}{\frac{4\left(\frac{\sin^3\beta}{\kappa_1(\boldsymbol z_3)}+\frac{\sin^3\beta}{\kappa_1(\boldsymbol z_3)}\right)}{\|\boldsymbol c_2\|^3}}=\frac{\sin^3\beta}{\kappa_3(\boldsymbol y)}$$
(cf. \fref{fig:06}). But it should be strongly emphasized that the above equality implies neither that $K$ is a body of flotation for $I^{\hat\delta}(K)$ nor that $\varphi_{I^{\hat\delta}(K)}(K)$ is a body of illumination for $I^{\hat\delta}(K)$. Further, this time it does not even imply that the affine arc length of the boundary cut off from $I^{\hat\delta}(K)$ by the chord $[\boldsymbol x,\boldsymbol y]$ is constant. A direct computation shows that the line tangent to $\varphi_{I^{\hat\delta}(K)}(K)$ at $\boldsymbol z$ is parallel to the line $\boldsymbol z_1\boldsymbol z_2$ (sic!), which in general does not need to be parallel to the line $\boldsymbol{xy}$.
\end{remark}

Following the lines of \sref{sec:05}, one can iterate the procedure of differentiating each successive higher-order curvature of the boundary of $K$ and substituting \eqref{eq:14} to obtain a complete system for differential invariants for the planar Euclidean group, constraining the derivatives of arbitrary order. Expressions quite quickly become lengthy, so computer assistance is badly needed. Experience shows that substituting new variables for $\det(\boldsymbol c,\gamma'(s))$ and $\det(\boldsymbol c,\gamma'(t))$ simplifies the formulas significantly, but still we were unable to find such a nice geometric interpretation of any subsequent result.

\subsection{The weak homothety conjecture}

Locally, one can construct two arcs that satisfy the conclusions of both \thref{thm:01,thm:02} (cf. \cite[Remark at the end of \S 2]{Bracho2004}), so any solution to the homothety conjecture really requires a global argument, e.g., some kind of billiard (cf. \cite{10.1112/blms/bdm068}). Let us therefore consider a weak variant of the same conjecture, when all three bodies associated with $K$: the body of flotation $F_\delta$, the body of buoyancy $B_\delta$ and the dual body of illumination $I^{\hat\delta}$ are assumed to be pairwise homothetic (possibly with respect to different centers). Note that such a question makes sense only because we know \emph{a posteriori} the value of $\hat\delta$.

\begin{theorem}\thlabel{thm:05}
Let $K\subset\mathbb R^2$ be a convex body with boundary of class $C^2$, and let $\delta\in(0,\Vol_2(K))$. If $\Pi_\delta(K)$, $\Gamma_\delta(K)$, and $\Pi^{\hat\delta}(K)$ are pairwise homothetic (possibly with respect to different centers), then $K$ is an ellipse.
\end{theorem}

\begin{proof}
Let $\boldsymbol r_1,\boldsymbol r_2,\boldsymbol r_3$ be the standard parametrizations of $\Pi_\delta(K),\Gamma_\delta(K),\Pi^{\hat\delta}(K)$, respectively. Since $\dot{\boldsymbol r}_1,\dot{\boldsymbol r}_2,\dot{\boldsymbol r}_3$ are pairwise parallel, points $\boldsymbol r_1(s),\boldsymbol r_2(s),\boldsymbol r_3(s)$ must correspond to each other in homothety for every $s\in I$. By \thref{thm:03}, the line $\boldsymbol r_1\boldsymbol r_3$ is parallel to the affine normal to $\Gamma_\delta(K)$ at $\boldsymbol r_2$. But since $\Pi_\delta(K)$ and $\Gamma_\delta(K)$ are homothetic, it is precisely the affine normal to $\Pi_\delta(K)$ at $\boldsymbol r_1$. Further, since $\Pi_\delta(K)$ and $\Pi^{\hat\delta}(K)$ are homothetic, the affine normals to $\Pi_\delta(K)$ meet at one point, which is the center of homothety. Hence, $\Pi_\delta(K)$ is a proper affine hypersphere, which by \thref{thm:04} must be an ellipse. Finally, by \cite[Theorem~3.2]{Kurusa2015}, $K$ itself must be an ellipse, which concludes the proof.
\end{proof}

Now, let's further assume that $K$ is centrally symmetric, and let $\delta\to\frac{1}{2}\Vol_2(K)$. Then $\hat\delta\to\infty$, whence both $\Pi_\delta(K)$ and $\Pi^{\hat\delta}(K)$ converge, after normalization, to $\mathcal I^*K$ (cf. \sref{sec:07}). It follows that they are, in a sense, \enquote{asymptotically homothetic}, but formally they do not meet the assumptions of \thref{thm:05}. Hence, it begs to consider also its limiting case:

\begin{theorem}\thlabel{thm:06}
Let $K\subset\mathbb R^2$ be a symmetric convex body with boundary of class $C^2$, and let $\delta_*=\frac{1}{2}\Vol_2(K)$. If $\Gamma_{\delta_*}(K)$ is homothetic to $\mathcal I^*K$, then $K$ is an ellipse.
\end{theorem}

\begin{proof}
We simply repeat the proof of \thref{thm:05}. Since both $\dot{\boldsymbol r}_2,\gamma_*'$ are parallel to $\gamma$, points $\boldsymbol r_2(s),\gamma_*(s)$ correspond to each other in homothety for every $s\in I$. By \thref{thm:03}, the affine normal to $\Gamma_{\delta_*}(K)$ at $\boldsymbol r_2$ is the line parallel to $\gamma'$ (and thus also to $\gamma_*$) passing through the origin. But since $\mathcal I^*K$ and $\Gamma_{\delta_*}(K)$ are homothetic, it is precisely the affine normal to $\mathcal I^*K$ at $\gamma_*$. Hence, $\mathcal I^*K$ is a proper affine hypersphere, which by \thref{thm:04} must be an ellipse. Finally, by \cite[Corollary~8.1.7]{Gardner_2006}, $K$ itself must be an ellipse, which concludes the proof.
\end{proof}

The above argument also gives an alternative proof of Petty's characterization of ellipses.

\begin{definition}[{cf. \cite[\S 6]{Petty1985}}]
Let $h_K$ be the supporting function of a convex body $K\subset\mathbb R^n$ with respect to its Santal\'o point $s(K)$. Let $f=ch_K^{-n-1}$, where $c>0$ is a positive constant. By \cite[(3.1) and (5.1)]{Petty1985}, there exists a convex body $K_*$ with curvature function $f$. The class of all convex bodies $K$ such that $K$ and $K_*$ are homothetic will be denoted by $\mathscr F_n$. The ellipsoids belong to $\mathscr F_n$.
\end{definition}

\begin{theorem}[{\cite[Lemma~8.1]{Petty1985}}]
The class $\mathscr F_2$ consists only of ellipses.
\end{theorem}

\begin{proof}
The homothety assumption of \thref{thm:06} may be rephrased in an equivalent way as $\dot{\boldsymbol r}_2=\lambda\gamma_*'$ for some $\lambda>0$. From \eqref{eq:10} and \eqref{eq:26} we obtain
$$\gamma\cdot\left(-\frac{2}{3\delta}\det(\gamma,\gamma')\right)=\gamma\cdot\left(-\lambda\frac{\det(\gamma',\gamma'')}{2\det(\gamma,\gamma')^2}\right),$$
whence
$$\frac{\det(\gamma,\gamma')^3}{\det(\gamma',\gamma'')}=\frac{3\delta\lambda}{4}=\Const.$$
Now, observe that this is precisely the condition \eqref{eq:23}. The fact that it characterizes ellipses follows immediately from \thref{thm:06}.
\end{proof}

\begin{remark}
In the limiting case of the original homothety conjecture, we would assume that $K$ is homothetic to $\mathcal I^*K$, which is equivalent to the classical problem formulated by H. Busemann and C.M. Petty \cite[Problem~5]{BusemannPetty}. In terms of our parametrization, it reads as follows: if $t=t(s)$ satisfies $\gamma(t)\parallel\gamma'(s)$, then
$$\gamma(t)=\lambda\frac{\gamma'(s)}{2\det(\gamma(s),\gamma'(s))}$$
for some $\lambda>0$. Differentiating the above equation yields
$$\gamma'(t)\cdot\frac{\mathrm dt}{\mathrm ds}=\gamma(s)\cdot\left(-\lambda\frac{\det(\gamma'(s),\gamma''(s))}{2\det(\gamma(s),\gamma'(s))^2}\right).$$
In particular, $\gamma'(t)\parallel\gamma(s)$. It follows that $K$ is homothetic to $\mathcal I^*K$ if and only if its boundary is a Radon curve \cite{Radon}. Hence, the answer to the limiting case of \thref{que:01} is negative, unlike the answer to the limiting case of \thref{que:02} (cf. \cite[Problem~8]{BusemannPetty}) and \thref{que:03} (cf. \thref{thm:06}). However, dimension $n=2$ may be somehow special, as it was already remarked by Busemann and Petty \cite[Problem~5]{BusemannPetty}, and by E. Milman, S. Shabelman, and A. Yehudayoff \cite[Remark~1.3]{milman2024}, who recently solved an analogous problem in arbitrary dimension $n\geq 3$, assuming that $K$ is homothetic to $\mathcal IK$ \cite[Corollary~1.2]{milman2024}.
\end{remark}

\subsection{Carousels and rational perimetral density}

While giving a negative answer to the floating body problem \cite[19. Problem: Ulam]{scottish} for density $\frac{1}{2}$ on a plane, H. Auerbach proved some remarkable properties of bodies that float in equilibrium in every orientation. Namely, that the chords of flotation have constant length \cite[\S 3]{Auerbach1938} and cut off from the boundary an arc of constant length \cite[\S 4]{Auerbach1938}. Now, observe that \thref{thm:01} and \thref{thm:02} are precisely the affine counterparts of these two facts, respectively. It is as interesting as it is unanticipated, because we obtained our theorems from completely different assumptions. J. Bracho, L. Montejano, and D. Oliveros \cite{Bracho2004} proved that the circle is the only body that floats in equilibrium in every orientation with perimetral density $\frac{1}{3}$, $\frac{1}{4}$, $\frac{1}{5}$ and $\frac{2}{5}$. We will prove an affine analog of their result for affine perimetral density $\frac{1}{3}$.

\begin{theorem}\thlabel{thm:07}
Let $K\subset\mathbb R^2$ be a convex body with boundary of class $C^2$, and let $\delta\in(0,\Vol_2(K))$. If $\Pi_\delta(K)$ is homothetic to $\Gamma_\delta(K)$, and if $\delta$ is such that every chord of flotation cuts off from the boundary exactly $\frac{1}{3}$ of its total affine arc length, then $K$ is an ellipse.
\end{theorem}

\begin{figure}
\includegraphics{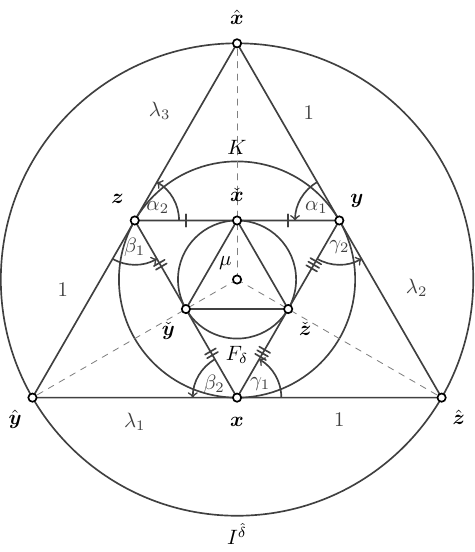}
\caption{Notation used in the proof of \thref{thm:07}}
\label{fig:07}
\end{figure}

\begin{proof}
We adopt the notation from \fref{fig:07}. Define also
$$\lambda_1\colonequals\frac{\hat{\boldsymbol y}\boldsymbol x}{\boldsymbol x\hat{\boldsymbol z}},\quad\lambda_2\colonequals\frac{\hat{\boldsymbol z}\boldsymbol y}{\boldsymbol y\hat{\boldsymbol x}},\quad\lambda_3\colonequals\frac{\hat{\boldsymbol x}\boldsymbol z}{\boldsymbol z\hat{\boldsymbol y}},$$
and observe that $\lambda_1\lambda_2\lambda_3=1$. Indeed, the law of sines yields
$$\lambda_1\lambda_2\lambda_3=\frac{\hat{\boldsymbol y}\boldsymbol x}{\boldsymbol x\hat{\boldsymbol z}}\frac{\hat{\boldsymbol z}\boldsymbol y}{\boldsymbol y\hat{\boldsymbol x}}\frac{\hat{\boldsymbol x}\boldsymbol z}{\boldsymbol z\hat{\boldsymbol y}}=\frac{\hat{\boldsymbol x}\boldsymbol z}{\boldsymbol y\hat{\boldsymbol x}}\frac{\hat{\boldsymbol y}\boldsymbol x}{\boldsymbol z\hat{\boldsymbol y}}\frac{\hat{\boldsymbol z}\boldsymbol y}{\boldsymbol x\hat{\boldsymbol z}}=\frac{\sin\alpha_1}{\sin\alpha_2}\frac{\sin\beta_1}{\sin\beta_2}\frac{\sin\gamma_1}{\sin\gamma_2}\overset{\eqref{eq:13}}{=}\frac{\kappa(\boldsymbol y)^{\frac{1}{3}}}{\kappa(\boldsymbol z)^{\frac{1}{3}}}\frac{\kappa(\boldsymbol z)^{\frac{1}{3}}}{\kappa(\boldsymbol x)^{\frac{1}{3}}}\frac{\kappa(\boldsymbol x)^{\frac{1}{3}}}{\kappa(\boldsymbol y)^{\frac{1}{3}}}=1.$$
Further, since $\Vol_2(\triangle\hat{\boldsymbol x}\boldsymbol{zy})=\Vol_2(\triangle\hat{\boldsymbol y}\boldsymbol{xz})=\Vol_2(\triangle\hat{\boldsymbol z}\boldsymbol{yx})$, we have
$$\frac{\lambda_3}{(\lambda_2+1)(\lambda_3+1)}=\frac{\lambda_1}{(\lambda_3+1)(\lambda_1+1)}=\frac{\lambda_2}{(\lambda_1+1)(\lambda_2+1)}.$$
Computing $\lambda_3$ from the first equation and substituting the result into the second equation, after simplification, reads $(\lambda_1-\lambda_2)(\lambda_1\lambda_2+\lambda_1+1)=0$. But since $\lambda_1,\lambda_2>0$, the only possibility is $\lambda_1=\lambda_2$. Hence, by symmetry, we have $\lambda_1=\lambda_2=\lambda_3$. When combined with the previous result, it immediately yields $\lambda_1=\lambda_2=\lambda_3=1$. Therefore, by the Second Theorem of Dupin, $\triangle\check{\boldsymbol x}\check{\boldsymbol y}\check{\boldsymbol z}$ is the medial triangle of $\triangle\boldsymbol x\boldsymbol y\boldsymbol z$, which in turn, by the above argument, is the medial triangle of $\triangle\hat{\boldsymbol x}\hat{\boldsymbol y}\hat{\boldsymbol z}$. In particular, all three triangles are pairwise homothetic with respect to the common centroid.\\

Following the idea of J. Bracho, L. Montejano, and D. Oliveros, we will show that the centroid of $\triangle\boldsymbol{xyz}$ is a fixed point (cf. \cite[Theorem~3]{Bracho2004}). Define $\mu\colonequals\frac{1}{3}(\boldsymbol x+\boldsymbol y+\boldsymbol z)$. It is enough to show that $\mu'(s)=\boldsymbol 0$ for every $s\in I$. Let $\gamma:I\to\mathbb R^2$ be a regular parametrization of $\partial K$, and let $t=t(s),u=u(s)$ be chosen so that $\boldsymbol x(s)=\gamma(s),\boldsymbol y(s)=\gamma(t(s)),\boldsymbol z(s)=\gamma(u(s))$. From \eqref{eq:14} we compute
\begin{equation}\label{eq:31}\frac{\mathrm dt}{\mathrm ds}=-\frac{\det(\gamma(t)-\gamma(s),\gamma'(s))}{\det(\gamma(t)-\gamma(s),\gamma'(t))},\quad\frac{\mathrm du}{\mathrm ds}=-\frac{\det(\gamma(u)-\gamma(s),\gamma'(s))}{\det(\gamma(u)-\gamma(s),\gamma'(u))}.\end{equation}
Further, because the sides of $\triangle\boldsymbol x\boldsymbol y\boldsymbol z$ are parallel to the sides of $\triangle\hat{\boldsymbol x}\hat{\boldsymbol y}\hat{\boldsymbol z}$, which in turn are tangent to $\partial K$ at the vertices of $\triangle\boldsymbol x\boldsymbol y\boldsymbol z$, we have
$$\det(\gamma(u)-\gamma(t),\gamma'(s))=0,\quad\det(\gamma(s)-\gamma(u),\gamma'(t))=0,\quad\det(\gamma(t)-\gamma(s),\gamma'(u))=0,$$
whence
\begin{equation}\label{eq:32}\gamma'(s)_2=\frac{\gamma(u)_2-\gamma(t)_2}{\gamma(u)_1-\gamma(t)_1}\gamma'(s)_1,\quad\gamma'(t)_2=\frac{\gamma(s)_2-\gamma(u)_2}{\gamma(s)_1-\gamma(u)_1}\gamma'(t)_1,\quad\gamma'(u)_2=\frac{\gamma(t)_2-\gamma(s)_2}{\gamma(t)_1-\gamma(s)_1}\gamma'(u)_1.\end{equation}
Substituting \eqref{eq:31} and \eqref{eq:32} to the formula for $\mu'$ yields $\boldsymbol 0$.\\

Finally, observe that $\hat{\boldsymbol x}\in\Pi^{\hat\delta}(K)$ is the image of $\check{\boldsymbol x}\in\Pi_\delta(K)$ under homothety with fixed center $\mu$ and fixed ratio $4$. Since we have this for every point $\hat{\boldsymbol x}\in\Pi^{\hat\delta}(K)$, it follows that the entire curve $\Pi^{\hat\delta}(K)$ is homothetic to $\Pi_\delta(K)$ with respect to the center $\mu$ and with ratio $4$. By \thref{thm:05}, this concludes the proof.
\end{proof}

Unfortunately, for perimetral density $\frac{1}{4}$ we were no longer able to find such a synthetic argument, since in our case the geometry is much less rigid than in \cite[Theorem~5]{Bracho2004}. While there are $6$ degrees of freedom, we have only $4$ equations, and thus the carousel does not need to be a parallelogram anymore.

\begin{remark}
For any rational perimetral density $\delta=\frac{p}{q}$, we can define functions $s=t_0(s),t_1(s),t_2(s),\ldots,t_q(s)$ being the coordinates of the vertices of a carousel with $q$ chairs. The system of linear equations
$$\det(\gamma(t_i)-\gamma(t_{i-1}),\gamma'(t_i)+\gamma'(t_{i-1}))=0,\quad i=1,2,\ldots,q$$
yields the expressions for $t_i'(s)$. Further, the system of linear equations
$$\frac{\mathrm d}{\mathrm ds}\frac{\det(\gamma'(t_{i-1}),\gamma(t_i)-\gamma(t_{i-1}))\det(\gamma(t_i)-\gamma(t_{i-1}),\gamma'(t_i))}{\det(\gamma'(t_{i-1}),\gamma'(t_i))}=0,\quad i=1,2,\ldots,q$$
yields the expressions for $\gamma''(t_i)_2$. Now, our first equation $E_1$ is $t_q(s)=t_0(s)=s$. It depends only on first-order derivatives of $\gamma$ at points $t_i$. We obtain $E_2$ by differentiating $E_1$ with respect to $s$. After elimination of $t_i'(s)$ and $\gamma''(t_i)$ (sic!), it again depends only on first-order derivatives of $\gamma$ at points $t_i$, and second-order derivatives of $\gamma$ at $s$. By iterating this algorithm, we should be able to obtain sufficiently many equations to eliminate all variables $\gamma(t_i)$, $i=1,2,\ldots,q$, thus obtaining a differential equation for $\gamma$ at $s$. However, expressions quickly become incomprehensible even for a computer, and elimination of $\gamma''(t_i)$ requires using other equations $E_i$, which are non-linear. Also, the final equation will most likely be of high order.
\end{remark}

\section*{Acknowledgments}

We would like to thank Prof. D.~Ryabogin for getting us interested in floating bodies, sharing his extensive knowledge with us, making us aware of the more profound context of our results, and the exceptional amount of effort spent on editorial work, Prof. M.~Angeles Alfonseca for helping us write the historical section on floating bodies and many inspiring discussions that eventually became the foundation for our work, Prof. S.~Myroshnychenko for bringing the bodies of illumination to our attention, and Prof. F.~Nazarov for numerous invaluable remarks and comments that helped us improve the paper.

\bibliography{references}{}
\bibliographystyle{amsplain}

\end{document}